\documentclass[11pt, reqno]{amsart}
\usepackage[T1]{fontenc}
\usepackage[latin1]{inputenc}
\usepackage[english]{babel}
\usepackage{amsthm}
\usepackage{amssymb,amsmath,bbm,dsfont}
\usepackage{enumitem}
\usepackage{enumitem}
\usepackage[all]{xy}
\usepackage{mathrsfs}
\usepackage[left=0.0cm,right=0.0cm,top=1.7cm,bottom=1.7cm]{geometry}
\usepackage{hyperref}
\usepackage{verbatim}

\theoremstyle{plain}
\newtheorem{theorem}{Theorem}[section]

\newtheorem{lemma}[theorem]{Lemma}
\newtheorem{proposition}[theorem]{Proposition}
\newtheorem{prop}[theorem]{Proposition}
\newtheorem{corollary}[theorem]{Corollary}
\newtheorem{cor}[theorem]{Corollary}

\theoremstyle{definition}
\newtheorem{definition}[theorem]{Definition}

\theoremstyle{remark}
\newtheorem{remark}[theorem]{Remark}

\theoremstyle{plain}

\def\C{{\bf C}}

\def\H{{H}}

\def\SL{{\mathbf{SL}}}

\def\A{{\bf A}}

\def\epsilon{\varepsilon}

\def\G{\mathbf{G}}
\def\H{\mathbf{H}}

\def\GSp{{\mathrm{GSp}}}
\def\Sp{{\mathrm{Sp}}}
\def\GU{{\mathrm{GU}}}

\def\GL{\mathrm{GL}}

\usepackage[usenames,dvipsnames]{color}

\title{A two variable Rankin--Selberg integral for $\GU(2,2)$ and the degree 5 $L$-function of $\GSp_4$}

\author{Antonio Cauchi}
\address{Antonio Cauchi\newline Dept. of Mathematics, School of Science\\ Tokyo Institute of Technology\\2-12-1 Ookayama, Meguro-ku\\ Tokyo, 152-8551 Japan}
\email{cauchi.a.aa@m.titech.ac.jp}

\author{Armando Gutierrez Terradillos}
\address{Armando Gutierrez Terradillos\newline 
Morningside Center of Mathematics, Chinese Academy of Sciences \\  
No. 55, Zhongguancun East Road, Haidian District, Beijing 100190, China}
\email{armando@amss.ac.cn}


\makeatletter
\def\@tocline#1#2#3#4#5#6#7{\relax
  \ifnum #1>\c@tocdepth 
  \else
    \par \addpenalty\@secpenalty\addvspace{#2}%
    \begingroup \hyphenpenalty\@M
    \@ifempty{#4}{%
      \@tempdima\csname r@tocindent\number#1\endcsname\relax
    }{%
      \@tempdima#4\relax
    }%
    \parindent\z@ \leftskip#3\relax \advance\leftskip\@tempdima\relax
    \rightskip\@pnumwidth plus4em \parfillskip-\@pnumwidth
    #5\leavevmode\hskip-\@tempdima
      \ifcase #1
       \or\or \hskip 1em \or \hskip 2em \else \hskip 3em \fi%
      #6\nobreak\relax
    \dotfill\hbox to\@pnumwidth{\@tocpagenum{#7}}\par
    \nobreak
    \endgroup
  \fi}
\makeatother
\usepackage{hyperref}

\begin{document}

\begin{abstract}
We give a two-variable Rankin--Selberg integral for generic cusp forms on $\mathrm{PGL}_4$ and $\mathrm{PGU}_{2,2}$ which represents a product of exterior square $L$-functions. As a residue of our integral, we obtain an integral representation on $\mathrm{PGU}_{2,2}$ of the degree 5 $L$-function of $\GSp_4$ twisted by the quadratic character of $E/F$ of cuspidal automorphic representations which contribute to the theta correspondence for the pair $(\mathrm{P}\GSp_4,\mathrm{P}\GU_{2,2})$.
 \end{abstract}

\maketitle
\selectlanguage{english}

\setcounter{tocdepth}{1}
\tableofcontents

\section{Introduction}

Automorphic $L$-functions are one of the cornerstones of the Langlands program. They provide a bridge between number theory and representation theory, dispensing a profound link between the arithmetic of algebraic objects such as elliptic curves and harmonic analysis on reductive groups. Automorphic $L$-functions are defined by the data of a cuspidal automorphic representation $\pi$ on a reductive group $G$ and of a finite dimensional representation $r$ of the Langlands group $^LG$. Their zeros or poles are often linked to the non-vanishing of certain periods of $\pi$ and to the participation of $\pi$ as a Langlands functorial transfer or theta lift of an automorphic representation of another reductive group. 
One of the main approaches to study the analytic properties of automorphic $L$-functions are via Rankin--Selberg integrals. Among those, the exotic family of Rankin--Selberg integrals in more than one complex variable represent products of $L$-functions and have applications to the study of their simultaneous poles. For instance, in \cite{BFG} Bump, Friedberg, and Ginzburg gave integrals in two complex variables on $\GSp_4$, $\GSp_6$, and $\GSp_8$, each representing the product of the spin and the standard $L$-functions. Similarly, in \cite{PollackShahMVGU} Pollack and Shah gave a two-variable integral on $\GU(2,2)$, which represents the product of the standard and the exterior square $L$-functions. In this article, we consider a new two-variable Rankin--Selberg integral, which for a globally generic cusp form on $\GU(2,2)$ or $\GL_4$, represents a product of exterior square $L$-functions.  As a residue of this integral, we obtain an integral representation of a twist of the degree $5$ $L$-function of $\GSp_4$, which gives a new instance of the curious phenomenon where a Rankin--Selberg integral of a cusp form on one group is used to represent an $L$-function of a different cusp form on another group. 
\\

We now describe our main results in more detail. Let $F$ be a number field and denote by $\A$ its adeles; we also let $E/F$ be a quadratic field extension of $F$ which defines the unitary group $\GU_{2,2/F}$.  We let $E^*_{B}(g,s,z)$ be the normalized two variable Borel Eisenstein series for $\GSp_{4/F}$. If $\Pi$ is a generic cuspidal automorphic representation of either $\GU_{2,2/F}$ or $\GL_{4/F}$ with trivial central character, and $\varphi$ is a cusp form in the space of $\Pi$, we define
\[I^*(\varphi, s ,z) := \int_{\GSp_{4}(F) Z_{\GSp_{4}}(\A)\setminus \GSp_{4}(\A)} \!\!\!\!\!E^*_{B}(g,s,z)\varphi(g)dg.\]
Let $\Sigma$ be a finite set of places of $F$ containing the ramified places for $\Pi$, for $E/F$ (when working with $\GU_{2,2}$), and the archimedean places of $F$. In Theorem \ref{mainunramified} and \eqref{mainunramifiednormalized}, we prove the following.

\begin{theorem}\label{mainthm1}
   We have \[I^*(\varphi, s ,z) = L^\Sigma(s, \Pi, \wedge^2)L^\Sigma(z, \Pi, \wedge^2) I_\Sigma(\varphi,s,z), \]
   where $I_\Sigma(\varphi,s,z)$, which is the product of the local integrals at places in $\Sigma$, can be made non-zero at $(s,z)$ for a certain choice of $\varphi$.
\end{theorem}

Let now $\Pi$ be a globally generic cuspidal automorphic representation of $\mathrm{PGU}_{2,2}(\A)$ and let $E^*_P(g,s)$ be the normalized degenerate Siegel Eisenstein series for ${\rm GSp}_4$. As $E^*_P(g,s)$ can be obtained as the residue at $z=1$ of $E^*_B(g,s,z)$, the residue at $z=1$ of $I^*(\varphi,s,z)$ is the integral
\[J^*(\varphi,s) = \int_{{\rm GSp}_4(F)Z_{{\rm GSp}_4}(\A) \backslash {\rm GSp}_4(\A)} E^*_P(g,s)\varphi(g)dg.\] 
We note that $J^*(\varphi,s) $ is the quasi-split counterpart of the integral for the group ${\rm SO}_{3,3}$ (hence $\mathrm{PGL}_4$) studied in \cite{BFGsplitorthogonal} and can be regarded as the degenerate companion of the integral for ${\rm SO}_{4,2}$ studied in \cite{Sugano} and \cite{PollackUGJ}.
 $J^*(\varphi,s)$ unfolds to the Shalika period of $\varphi$, which is non-zero if and only if the partial twisted exterior square $L$-function $L(s, \Pi, \wedge^2)$ has a simple pole at $s=1$ (\textit{cf}. \cite{FurusawaMorimoto}). When any of these hold, $\Pi$ comes from a globally generic cuspidal automorphic representation $\sigma$ of ${\rm GSp}_{4}(\A)$ with trivial central character, by means of a theta correspondence for $({\rm PGSp}_4,{\rm PGU}_{2,2})$ (\textit{cf}. \cite[Theorem B]{morimoto}), and one has that
\[ L^\Sigma(s,\Pi, \wedge^2) = L^\Sigma(s,\sigma, {\rm std} \otimes \chi_{E/F}) \zeta_F^\Sigma(s),\]
where $L^\Sigma(s,\sigma, {\rm std} \otimes \chi_{E/F})$ denotes the partial degree 5 $L$-function for ${\rm GSp}_{4}$ twisted by the quadratic Hecke character associated to the quadratic extension $E/F$ which defines ${\rm GU}_{2,2}$.  Theorem \ref{mainthm1} and an analysis of the integral $J^*(\varphi,s)$ at the points where $E^*_P(g,s)$ might obtain a pole let us show the following (\textit{cf}. Theorems \ref{Thisismain2} and \ref{polesofintegral}). 

\begin{theorem}\label{mainthm2} Suppose that $\Pi$ appears in the theta lift of $\sigma$ in the sense of \cite[Theorem B]{morimoto}, where $\sigma$ is a globally generic irreducible cuspidal automorphic representation of ${\rm GSp}_4(\A)$ with trivial central character. For a given cusp form $\varphi$ in $\Pi$, there exist a cusp form $\varphi'$ in $\Pi$ and section $f$ defining the Borel Eisenstein series such that 
 \[J^*(\varphi,   s) = L^\Sigma(s,\sigma, {\rm std} \otimes \chi_{E/F}) \cdot \prod_{v \in \Sigma_\infty}
     I(\varphi_v',f_v, s  , 1 ). \]
where $\Sigma$ is a finite set of places containing the set $\Sigma_\infty$ of archimedean places and the ramified places for $\Pi$,  $\sigma$,  and $E/F$. Furthermore, the integral $J^*(\varphi,   s)$ and the partial $L$-function $L^\Sigma(s,\sigma, {\rm std} \otimes \chi_{E/F})$ extend to holomorphic functions for all $s$.
\end{theorem}

\noindent One novelty of our integral representation $J^*(\varphi,s)$ of the degree 5 $L$-function of $\mathrm{GSp}_{4}$ is that it might be used to address motivic questions. Indeed, since $\mathrm{GSp}_{4}$ and $\mathrm{GU}_{2,2}$ have Shimura varieties, our integral seems to have applications to the study of the conjectural ``standard'' degree 5 motive associated to cohomological cuspidal automorphic representations of $\mathrm{GSp}_{4}$, such as the construction of an Euler system for the corresponding Galois representation, similar in spirit to the one given in work in progress by Skinner and Sangiovanni Vincentelli for the symmetric square of a modular form. 

\subsection*{Acknowledgements}
A.C. was financially supported by the JSPS Postdoctoral Fellowship for Research in Japan. A.G. was supported by the Morningside Center of Mathematics (CAS). We are greatly indebted to the anonymous referees, whose insights and comments have notably strengthened the main results of the manuscript and improved their exposition.  

\section{Groups and \texorpdfstring{$L$}{L}-functions}

Let $F$ be a number field, with ring of integers $\mathcal{O}$, and denote by $\A$ its adeles. We let $\zeta_F$ denote the Dedekind zeta function of $F$. When $v$ is a finite place of $F$, we let $F_v$ be the $v$-adic completion of $F$ with ring of integers $\mathcal{O}_v$; we let $q_v$ be the size of the quotient of $\mathcal{O}_v$ modulo its maximal ideal $\mathfrak{p}_v$ and fix a uniformizer  $\varpi_v$ of $\mathcal{O}_v$. Accordingly, we normalize the norm $|\cdot |$ of $F_v$ so that $|\varpi_v| = q_v^{-1}$. Here and throughout the article, if $G$ is a group with a parabolic subgroup $P$, we will denote by ${\rm Ind}_{P}^{G}(\bullet)$ the normalized induction and by by $I_P^G(\bullet)$ the not normalized one.

\subsection{Groups}
\subsubsection{Definitions}
Denote by $J_2$ the $2 \times 2$ anti-diagonal matrix with all entries 1 and $J=\left( \begin{smallmatrix}  & J_2 \\ -J_2 & \end{smallmatrix} \right)$.  Fix a quadratic extension $E/F$ and define $\G$ the group scheme over $F$ given by  
\[ \G(R):=\GU_{2,2}(R) = \{ (g,m_g) \in \GL_{4,F}(R \otimes_{F}  E ) \times \GL_{1,F}(R) :\bar{g}^t J g = m_g J\}, \]
where $\bar{\bullet}$ denotes the non-trivial automorphism of order 2 of $E/F$ and $R$ denotes any $F$-algebra.  Denote by $\nu:\G \to \GL_{1,F}, \; g \mapsto m_g$ the similitude character. Inside $\G$, we have the subgroup \[\H(R) := \GSp_{4,F}(R)=\{ (g,m_g) \in \GL_{4,F}(R) \times \GL_{1,F}(R) :g^t J g = m_g J\}, \]
which embeds into $\G$ via the natural inclusion. Throughout the manuscript, for any algebraic group ${\rm G}$, we let ${\rm PG}$ denote the quotient of ${\rm G}$ by its center.

\subsubsection{Parabolic groups}\label{parabolics}
Let $B_\G=T_\G U_\G$ denote the upper-triangular Borel subgroup of $\G$, with $T_\G$ the diagonal torus \[\left \{\left(\begin{smallmatrix}a& &&\\&b&& \\ & &\nu \bar{b}^{-1} & \\ & &&\nu \bar{a}^{-1}\end{smallmatrix}\right)\;:  a,b\in {\rm Res}_{E/F}\GL_{1,F} ,\;\nu  \in \GL_{1,F}  \right \}.\]
The modular character  $\delta_{B_\G}: B_\G(F_v) \to \C$ is given by \[ \left(\begin{smallmatrix}a& &&\\&b&& \\ & &\nu \bar{b}^{-1} & \\ & &&\nu \bar{a}^{-1}\end{smallmatrix}\right) \mapsto \tfrac{|a\bar{a}|^3|b\bar{b}|}{|\nu|^4}.\]
Let $Q_\G=M_{Q_\G} N_{Q_\G}$ denote the Klingen parabolic of $\G$. It has Levi  $M_{Q_\G} \simeq \GU_{1,1} \times {\rm Res}_{E/F} \GL_1$. Similarly, we let $B_\H = T_\H U_\H = B_\G \cap \H$ be the upper triangular Borel subgroup of $\H$, with diagonal torus \[T_\H = \left \{t(a,b;\mu)=\left(\begin{smallmatrix}a& &&\\&b&& \\ & &\mu b^{-1} & \\ & &&\mu a^{-1}\end{smallmatrix}\right)\;:  a,b, \mu  \in \GL_{1,F} \right \}.\]
Let $Q_\H$ denote the standard Klingen parabolic of $\H$ with Levi decomposition $M_{Q_\H} N_{Q_\H}$ where \begin{equation*}
M_{Q_\H} = \left \{ \left(\begin{smallmatrix} a &  &  \\ & r & \\  & & d \end{smallmatrix} \right),\; r \in \GL_{2,F}, a,d \in \GL_{1,F}\;:\; ad = {\rm det}(r) \right \}, \end{equation*}
and $N_{Q_\H} \simeq \mathbf{G}_{a,F}^3$. Denote by $\delta_{Q_\H}$  the modulus character of ${Q_\H}$, given by \begin{align*}
\delta_{Q_\H}&:\left(\begin{smallmatrix} a &  &  \\ & r & \\  & & d \end{smallmatrix} \right)  \mapsto \left|\tfrac{a}{d}\right|^2.
\end{align*}
Let  $P_\H$ be the standard Siegel parabolic of $\H$ with Levi decomposition $M_{P_\H} N_{P_\H}$, with \begin{equation*}
    M_{P_\H} = \left \{ \left(\begin{smallmatrix} g & \\ & \mu J_2 {}^tg^{-1} J_2 \end{smallmatrix} \right),\; g \in \GL_{2,F}, \mu \in \GL_{1,F}\right \}, \end{equation*} and modulus character given by \begin{align*}
\delta_{P_\H}&:\left(\begin{smallmatrix} g & \star \\ & \mu J_2 {}^tg^{-1} J_2 \end{smallmatrix} \right)  \mapsto \left|\tfrac{{\rm det}(g)}{\mu}\right|^3.
\end{align*}

\subsection{Dual groups and $L$-functions}

\subsubsection{Dual groups}

Recall that the dual group of $\G$ is \[ {}^L\G = ( \GL_4(\C) \times \GL_1(\C)) \rtimes {\rm Gal}(E/F),  \]
with the action of the non-trivial element $\tau \in {\rm Gal}(E/F)$ given by (cf. \cite[\S 1.8(c)]{BlasiusRogawski})
\[ (g,\lambda) \mapsto (\Phi_4 {}^tg^{-1} \Phi_4 , \lambda {\rm det}(g)), \text{ where }  \Phi_4 = \left( \begin{smallmatrix}  & & &1\\&&-1& \\ & 1&&\\ -1& & & \end{smallmatrix} \right). \] 
The dual group of $\mathbf{P}\G$ is the derived subgroup of ${}^L\G$, namely
\[ {}^L\mathbf{P}\G = {\rm SL}_4(\C) \rtimes {\rm Gal}(E/F). \]

Recall also that $ ^L\mathbf{P}\H = {\rm Spin}_5(\C) \simeq \Sp_4(\C)$. Under this exceptional isomorphism, the standard representation $W_4$ of $\Sp_4(\C)$ coincides with the spin representation of ${\rm Spin}_5(\C)$.
 
\subsubsection{The twisted exterior square representation}

The standard representation of $\GL_4(\C)$ on $\C^4$, with basis $\{e_1,e_2,e_3,e_4\}$, induces the representation $\wedge^2$ on $\wedge^2 \C^4 \simeq \C^6$ by the rule $(\wedge^2 g) \cdot ( e_i \wedge e_j) = g \cdot e_i \wedge g \cdot e_j$. Consider the six dimensional representation \[ \GL_4(\C) \times \GL_1(\C) \to \GL_6(\C)\] given by $ (g,\lambda) \mapsto  (\wedge^2 g) \lambda$. By \cite[Lemma 2.1]{KimExterior}, $\wedge^2$ extends to a representation of ${}^L\G$. Indeed, notice that both representations $\wedge^2$ and $\wedge^2 \circ \tau$, with $1 \ne \tau \in {\rm Gal}(E/F)$, have same highest weight. Hence they are isomorphic, i.e. there exists $A \in \GL_6(\C)$ such that  $(\wedge^2 \circ \tau) (g,\lambda) = A^{-1} ( \wedge^2  (g,\lambda) ) A$. Choose $A$ with ${\rm Tr}(A) >0$ and extend $\wedge^2$ to ${}^L\G$ by sending \[ (g,\lambda,1) \mapsto (\wedge^2g)\lambda,\, (1,1,\tau) \mapsto A.\]
The resulting representation, which we still denote by $\wedge^2: {}^L\G \to \GL_6(\C)$, is called the (twisted) exterior square representation of ${}^L\G$. We still denote by $\wedge^2: {}^L\mathbf{P}\G \to \GL_6(\C)$ the composition of $ {}^L\mathbf{P}\G \hookrightarrow  {}^L\G$ with $\wedge^2$.
 
If we restrict the exterior square representation along the embedding ${}^L\mathbf{P}\H \hookrightarrow {}^L\mathbf{P}\G, \, g \mapsto (g,1)$, we have \begin{equation}\label{brwedge}
    \wedge^2_{|\Sp_4(\C)} = {\rm std} \oplus \mathbf{1},
\end{equation} where ${\rm std}: \Sp_4(\C) \to \GL_5(\C)$ is the irreducible representation of highest weight $(1,1)$ and gives the composition of the projection $\Sp_4(\C) \simeq {\rm Spin}_5(\C) \to {\rm SO}_5(\C)$ with the standard representation of ${\rm SO}_5(\C)$.

\subsubsection{$L$-factors}

Suppose that $v$ is an unramified place for $E/F$ and let $\Pi_v$ be an unramified representation of $\mathbf{P}\G(F_v)$ with Frobenius conjugacy class $g_{\Pi_v}$. Define the local Euler factor of the exterior square $L$-function for $\Pi_v$ by 
\[ L(s, \Pi_v, \wedge^2) : = {{\rm det}(1 - \wedge^2(g_{\Pi_v}) \cdot q_v^{-s})}^{-1}.\]
These Euler factors can be related to Euler factors of other $L$-functions on $\mathbf{PH}$ or $\mathrm{PGL}_4$, as follows.
\begin{itemize}
    \item Suppose that $v$ splits in $E$, then the isomorphism $E \otimes_F F_v \simeq F_v \times F_v$ induces an isomorphism $\GL_4(E \otimes_F F_v  ) \simeq \GL_4(F_v) \times \GL_4( F_v)$. If $g$ maps to $(g_1,g_2)$, then $\bar{g}$ is sent to $(g_2,g_1)$. Hence, we have \begin{align*}
         \G(F_v)  \simeq  \{ (g_1 , g_2, m) \in \GL_4(F_v) \times \GL_4( F_v) \times \GL_1(F_v)\,:\, g_2 = m J {}^tg_1^{-1} J^{-1}  \}  
    \end{align*}
Hence, projection to the first factor induces \begin{align}\label{isoforpsplit} \G(F_v)  \simeq  \GL_4( F_v ) \times \GL_1(F_v), \, (g_1, g_2, m) \mapsto (g_1,m).\end{align} 
This let us identify $\mathbf{P}\G(F_v)$ with $\mathrm{PGL}_4(F_v)$ and thus $\Pi_v$ with an unramified representation $\pi_v$ of $\GL_4(F_v)$ with trivial central character. Then \[ L(s, \Pi_v, \wedge^2) = L(s, \pi_v, \wedge^2).\]
Using the exceptional isomorphism $\mathrm{PGL}_4(F_v) \simeq \mathrm{PGSO}_{3,3}(F_v)$, we can consider the local theta correspondence for $(\mathbf{P}\H(F_v),\mathrm{PGL}_4(F_v))$. If $\pi_v$ is the small theta lift of a representation $\sigma_v$ of $\mathbf{P}\H(F_v)$, by functoriality of the theta lift and the decomposition \eqref{brwedge}, we have 
\[ L(s, \pi_v, \wedge^2) = L(s, \sigma_v, {\rm std}) \cdot  {\zeta}_v(s).\]
    \item Let $v$ be inert in $E$ and denote again with $v$ the corresponding place of $E$. If $\Pi_v$ is an unramified irreducible subquotient of the (normalized) induction $\mathrm{Ind}_{B_{\G}(F_v)}^{\G(F_v)}\xi$, denote by $\pi_v$ the unramified irreducible subquotient of $\mathrm{Ind}_{B_{\H}(F_v)}^{\H(F_v)}\xi',$ with $\xi'$ being the restriction of $\xi$ to $T_\H(F_v)$. Then, we have (\textit{cf}. \cite[Proposition 3.2]{PollackShahMVGU}) \begin{align}\label{extTOspin}
        L(s, \Pi_v, \wedge^2) = L(s, \pi_v, {\rm Spin}) \cdot  {\zeta}_v(2s).
    \end{align}
Finally, let \[\H(F_v)^+ := \{ (g,m_g) \in \H(F_v)\,:\, m_g \in N_{E_v/F_v}( E_v^\times) \}, \] 
and consider the local theta correspondence for $(\mathbf{P}\H(F_v)^+, \mathbf{P}\G(F_v))$ studied in \cite{morimoto}. If $\Pi_v$ is the small theta lift of a generic unramified representation $\sigma_v^+$ of $\mathbf{P}\H(F_v)^+$, and $\sigma_v^+$ appears as an irreducible factor of an unramified representation $\sigma_v$ of $\H(F_v)$, then we have (\textit{cf} \cite[Corollary 6.23]{morimoto})
\begin{align}\label{extTOstd}
        L(s, \Pi_v, \wedge^2) = L(s, \sigma_v, {\rm std} \otimes \chi_{E_v/F_v}) \cdot  {\zeta}_v(s),
    \end{align}
where \[ \chi_{E_v/F_v}(x) = \begin{cases} -1 & \text{if } x \not \in N_{E_v/F_v}(E_v^\times)  \\
 1 & \text{otherwise.}\end{cases}
\]
\end{itemize}

\section{The two variable Rankin--Selberg integral}

\subsection{Global integral and unfolding} 

Let $s,z$ be complex variables. We denote by $I_{B_\H}(s,z)$ the induced representation of $\H(\A)$ consisting of smooth functions $f(\cdot, s, z):\H(\A)\to \C$ such that, $\forall n \in U_{\H}(\A), \forall t \in T_{\H}(\A)$, we have
\begin{equation}\label{sectionEisenstein}f(ntg, s , z) = \chi_{s,z}(t) f(g,s,z),\end{equation}
for $\chi_{s,z}(t(a,b;\mu)):=|a|^{s+z+1} |b|^{s-z+1} |\mu|^{-s-1}$.
This representation has trivial central character. Moreover, the restriction to $P_\H$ of $\chi_{s,0}$ is $\delta_{P_\H}^{\tfrac{1}{3}(s+1)}$.

We let $f \in I_{B_\H}(s,z)$ be a standard section, which factorizes as $\otimes_v f_v$, and define the two variable Borel Eisenstein series  \begin{align}\label{KES}
    {\rm Eis}_{B_\H}(g,s,z) :=    {\sum_{\gamma \in {B_\H}(F) \backslash \H(F)}} f (\gamma g , s , z).
\end{align}  
The sum converges absolutely for ${\rm Re}(s)$ and ${\rm Re}(z)$ big enough and it extends to a meromorphic function on $\C^2$. Let $(\Pi,V_\Pi)$ be a cuspidal automorphic representation of $\mathbf{PG}(\A)$. Given a cusp form $\varphi\in V_\Pi$, we define
\[I(\varphi,f, s ,z) := \int_{\H(F) Z_\H(\A)\setminus \H(\A)} \!\!\!\!\!{\rm Eis}_{B_\H}(g,s,z)\varphi(g)dg.\]

Let $\psi$ be a non-trivial additive character of $\A/F$. We define a principal character $\chi:U_{\G}(\A)/U_{\G}(F) \to\C^{\times}$ as follows: \[\chi\left(\left(\begin{smallmatrix} 1 & x & y & a\\ & 1 & b & \bar{ y} \\ & & 1 & - \bar{x} \\ & & & 1\end{smallmatrix} \right)\right) = \psi(\mathrm{Tr}_{E/F}(\alpha x)-b),\]
where $\alpha \in E$ is such that $\overline{\alpha} = -\alpha \neq 0$. We consider the Whittaker model of a cusp form $\varphi$ of $\G$ with respect to $\chi$, i.e. 
\[W_{\varphi}(g) := \int_{[U_{\G}]}\varphi(ng)\chi(n)dn.\]
\begin{prop}\label{unfolding}
The zeta integral unfolds to 
\[ I(\varphi,f, s  , z ) = \int_{N_{Q_\H}(\A)Z_\H(\A)\setminus \H(\A)} \!\!\!\!\!\!f(w g , s  , z)   W_{\varphi}(g) dg,\]
where 
\[ w = \left ( \begin{smallmatrix}
    1 & & & \\ & & 1 & \\ & -1 & & \\ & & & 1
\end{smallmatrix} \right ).\]
\end{prop}

\begin{proof}

We start by unfolding the Eisenstein series along the Klingen parabolic $Q_{\H}$ to get 

\[ \int_{M_{Q_\H}(F) N_{Q_\H}(\A) Z_\H(\A)\setminus \H(\A)} \!\!\!\!\!\! \varphi_{N_{Q_\H}}(g) \cdot \!\!\!\!\!\! \sum_{ \gamma \in {B_\H}(F) \backslash Q_\H(F)} f(\gamma g , s , z) dg,\]
where $\varphi_{N_{Q_\H}}(g)$ denotes the constant term of $\varphi$ along $N_{Q_\H}$. We now Fourier expand $\varphi_{N_{Q_\H}}$ along $N_{Q_{\H}}\setminus N_{Q_{\G}}$ to get
\[\varphi_{N_{Q_\H}}(g) = \sum_{\chi:[N_{Q_{\H}}\setminus N_{Q_{\G}}]\to \C^{\times}}\varphi_{\chi}(g),\]
with $\varphi_{\chi}(g) = \int_{[N_{Q_{\G}}]}\varphi(n'g)\chi^{-1}(n')dn'$. Any character $ \chi : [N_{Q_{\H}}\setminus N_{Q_{\G}}] \to \C^\times$ can be described as follows. Recall that the characters $\chi: [N_{Q_{\G}} / [N_{Q_{\G}},N_{Q_{\G}}] ] \to \C^\times$ are of the form $n' \mapsto \psi( {\rm Tr}_{E/F} (\alpha x + \beta y)), $ where $\alpha, \beta \in E$, $\psi$ is a non-trivial additive character on $F \backslash \A$, and where we have written \[ n' \equiv  \left(\begin{smallmatrix} 1 & x & y & 0\\ & 1 &  & \bar{ y} \\ & & 1 & - \bar{x} \\ & & & 1\end{smallmatrix} \right)\;\text{ mod } [N_{Q_{\G}},N_{Q_{\G}}],\; \text{ with } x,y \in \A_E.\] 
Such a character $\chi$ is trivial on $N_{Q_\H}(\A)$ if $\alpha + \bar{\alpha} = \beta + \bar{\beta} = 0$. We deduce that any character $\chi: [ N_{Q_\H} \setminus N_{Q_\G} ] \to \C^\times$ is of the form $n' \mapsto  \psi( {\rm Tr}_{E/F} (\alpha x + \beta y))$, with $\alpha, \beta \in E$ such that $\bar{\alpha} = - \alpha$ and  $\bar{\beta} = - \beta$. To remark the dependence by these data, denote such a character by $\chi_{\alpha,\beta}$. 

\noindent Since $M_{Q_\H}(F)$ acts on $N_{Q_{\G}}(\A)$ by conjugation and this action preserves $N_{Q_\H}(\A)$, $M_{Q_\H}(F)$ acts on the space of characters of $[N_{Q_\H}\setminus N_{Q_{\G}}]$. Indeed, if $ m  =\left(\begin{smallmatrix} a &  &  \\ & g & \\  & & d \end{smallmatrix} \right) \in M_{Q_\H}(F)$ and $g^{-1} = (\tilde{g}_{i,j})$, \begin{align*}
\chi_{\alpha,\beta}( m n' m^{-1}) &= \psi ( a \cdot  {\rm Tr}_{E/F} (\alpha    (\tilde{g}_{1,1} x + \tilde{g}_{2,1} y)  + \beta  (\tilde{g}_{1,2} x + \tilde{g}_{2,2} y) ))\\
&= \psi ( {\rm Tr}_{E/F}( a(\alpha \tilde{g}_{1,1} + \beta \tilde{g}_{1,2} ) x +  a ( \alpha \tilde{g}_{2,1} + \beta \tilde{g}_{2,2}) y) )\\
&= \chi_{\alpha',\beta'}(  n' ),
\end{align*}  
with $\alpha' = a(\alpha \tilde{g}_{1,1} + \beta \tilde{g}_{1,2} )$ and $\beta' =  a ( \alpha \tilde{g}_{2,1} + \beta \tilde{g}_{2,2})$.
Then, $M_{Q_\H}(F)$ acts on the space of characters of $[N_{Q_\H}\setminus N_{Q_{\G}}]$ with two orbits, the trivial one and an open one. A representative of the open orbit is given by the character $\chi_{\alpha,0}$ with $\bar{\alpha} = - \alpha \ne 0$. As the constant term $\varphi_{N_{Q_{\G}}}(g)$ vanishes because of cuspidality of $\varphi$,  the integral can be written as
\begin{equation}\label{auxFourierexp} \int_{M_{\alpha}(F)N_{Q_{\H}}(\A)Z_\H(\A)\setminus \H(\A)} \!\!\!\!\!\!\varphi_{\chi_{\alpha,0}}(g) \cdot \!\!\!\!\!\! 
 \sum_{ \gamma \in {B_\H}(F) \backslash Q_\H(F)} f(\gamma g , s , z)  dg,\end{equation}
where $M_{\alpha}(F)$ is the stabilizer of $\chi_{\alpha,0}$ in $M_{Q_{\H}}(F)$: \[M_{\alpha}(F) =L_{\alpha}(F)N_{\alpha}(F) = \left \{  \left(\begin{smallmatrix} a &   &   &  \\ & a &   &   \\ & & d &   \\ & & & d\end{smallmatrix} \right)\left(\begin{smallmatrix} 1 &   &   &  \\ & 1 &  b &   \\ & & 1 &   \\ & & & 1\end{smallmatrix} \right)\;a,d\in F^\times,\, b \in F \right \}.\]

By the Bruhat decomposition for $\GL_2$, we have that $ {B_\H}(F) \backslash Q_\H(F) = {B_\H}(F) \cup  {B_\H}(F) w N_\alpha(F) $. This let us  write \eqref{auxFourierexp} as a sum of two integrals, with the term corresponding to the coset ${B_\H}(F)$ that vanishes because of the cuspidality of $\varphi$:  

\[ \eqref{auxFourierexp} = \int_{L_{\alpha}(F) N_{Q_{\H}}(\A)Z_\H(\A)\setminus \H(\A)} \!\!\!\!\!\!\varphi_{\chi_{\alpha,0}}(g) f( w g , s , z)  dg. \]
We now Fourier expand $\varphi_{\chi_{\alpha,0}}(g)$ over $ N_{Q_\G} \backslash U_\G$. $L_\alpha(F)$ acts on the space of characters on $[ N_{Q_\G} \backslash U_\G ]$ with one open orbit and the trivial one. The latter does not contribute to the integral because of cuspidality of $\varphi$. As $U_\G =N_{Q_\G} N_\alpha$, a representative of the open orbit can be taken to be $b \mapsto \psi(b)$ on $N_\alpha$ and has stabilizer $Z_\H(F)$.  Since the character on $ [U_\G /[U_\G, U_\G]]$, defined by taking $b \mapsto \psi(-b)$ on $N_\alpha$ and by taking $\chi_{\alpha,0}$ on the one parameter unipotent subgroup associated to the other simple root, is a principal character, the function $W_{\varphi}(g) = \int_{[N_{\alpha}]}\varphi_{\chi_{\alpha,0}}(n(b)g)\psi(-b)db$ is a Whittaker functional for $\varphi$ in  $\Pi$. This concludes the proof.
\end{proof}

\begin{remark}
Using the natural embedding of $\H \hookrightarrow \GL_4$, we can define the analogous integral for cusp forms on $\mathrm{PGL}_4$ by integrating their restriction to $\H(\A)$ against ${\rm Eis}_{B_\H}(g,s,z)$. Then the unfolding follows from an almost verbatim translation of the proof of Proposition \ref{unfolding}. We leave the details to the reader. 
\end{remark}

By the uniqueness of Whittaker functionals, we have the following.
\begin{cor}\label{Eulerproduct}
Let $\varphi = \otimes_v \varphi_v$ and $f=\otimes_v f_v$ be factorizable elements of $\Pi$ and $I_{B_\H}(s,z)$ respectively. Then $I(\varphi,f, s  , z ) = \prod_v I(\varphi_v,f_v, s  , z )$, with 
\[ I(\varphi_v,f_v, s  , z ) = \int_{N_{Q_\H}(F_v)Z_\H(F_v)\setminus \H(F_v)} \!\!\!\!\!\!f_v(w g_v , s  , z)   W_{\varphi_v}(g_v) dg_v.\]
\end{cor}

We are now ready to state our first main result. Let $(\Pi,V_\Pi)$ be a globally generic cuspidal automorphic representation of either $\mathbf{P}\G(\A)$ or $\mathrm{PGL}_4(\A)$ and let $\varphi = \otimes_v \varphi_v$ be a factorizable cusp form of $\Pi$. Let $\Sigma$ be a finite set of places of $F$ containing the ramified places for $\Pi$, for $E/F$ (when working with $\mathbf{P}\G$), and the archimedean places of $F$. Fix a standard section $f = \otimes_v f_v \in I_{B_\H}(s,z)$ such that $f_v$ is unramified at every $v \not \in \Sigma$ with value $f_v(1,s,z)=1$. 
\begin{theorem}\label{mainunramified}
 We have 
 \[I(\varphi,f, s  , z ) = \frac{L^\Sigma(s, \Pi, \wedge^2)L^\Sigma(z, \Pi, \wedge^2)}{\zeta_F^\Sigma(2z)\zeta_F^\Sigma(2s)\zeta_F^\Sigma(s+z)\zeta_F^\Sigma(s-z+1)}  \prod_{v \in \Sigma} I(\varphi_v,f_v, s  , z ).\]
 Moreover, for given $s,z \in \C$ we can choose $\varphi_v$ and $f_v$ so that $\prod_{v \in \Sigma} I(\varphi_v,f_v, s  , z )$ is non-zero. 
\end{theorem}
\begin{proof}
This follows from Corollary \ref{Eulerproduct}, Theorem \ref{SplitTheorem}, Theorem \ref{inertTheorem}, and Lemma \ref{LemmaonRamified}.
\end{proof}

Finally, if we normalize $f_v^*(g,s,z):= \zeta_v(2z)\zeta_v(2s)\zeta_v(s+z)\zeta_v(s-z+1) f_v(g,s,z)$ at every $v \not \in \Sigma$ and denote by $E^*_B(g,s,z)$ the corresponding normalized Borel Eisenstein series, we can define 
\[I^*(\varphi, s ,z) := \int_{\H(F) Z_\H(\A)\setminus \H(\A)} \!\!\!\!\!{E}^*_{B}(g,s,z)\varphi(g)dg.\]
Theorem \ref{mainunramified} shows that \begin{align}\label{mainunramifiednormalized}
    I^*(\varphi, s ,z) = L^\Sigma(s, \Pi, \wedge^2)L^\Sigma(z, \Pi, \wedge^2)\prod_{v \in \Sigma} I(\varphi_v,f_v, s  , z ).
\end{align}

\subsection{The unramified unipotent computation}
Let $v$ be a finite place of $F$ which is unramified for $\Pi$, $f$, and $E/F$. Recall that we have denoted $q_v$ the size of the quotient of $\mathcal{O}_v$ modulo its maximal ideal $\mathfrak{p}_v$ and fixed a uniformizer  $\varpi_v$ of $\mathcal{O}_v$. The norm $|\cdot |$ of $F_v$ is normalized so that $|\varpi_v| = q_v^{-1}$. Let $W_{v}$ denote the vector in the Whittaker model of $\Pi_v$ normalized so that it is trivial on $\G(\mathcal{O}_v)$; finally, recall that $f_v(1,s,z)=1$. Using the Iwasawa decomposition $\H(F_v) = B_{\H}(F_v) \H(\mathcal{O}_v)$ and identifying $N_{Q_\H}(F_v)Z_\H(F_v)\setminus B_{\H}(F_v)$ with 
\[L(F_v) := \left \{ \ell_v = n(b)m(a,d) := \left(\begin{smallmatrix} 1 &   &   &  \\ & 1 &  b &   \\ & & 1 &   \\ & & & 1\end{smallmatrix} \right)\left(\begin{smallmatrix} a d &   &   &  \\ & d &   &   \\ & & 1 &   \\ & & & a^{-1}\end{smallmatrix} \right)\;a,d\in F_v^\times,\, b \in F_v \right \},\]
\noindent we write \begin{align*}
    I(\varphi_v,f_v, s  , z ) &= \int_{L(F_v)} \delta^{-1}_{B_{\H}}(\ell_v) f_v(w \ell_v , s  , z)   W_v(\ell_v) d\ell_v \\ 
    &= \int_{(F_v^\times)^2} |a|^{-4}|d|^{-3} W_v(m(a,d))\int_{ F_v } f_v(w \,n(b )  m(a,d) , s  , z) \psi_v(b) db \, d^\times a \,d^\times d \\
    &= \int_{(F_v^\times)^2} |a|^{-4}|d|^{-2} W_v(m(a,d))\int_{ F_v } f_v(w \, m(a,d) n(b ), s  , z) \psi_v(bd) db \, d^\times a \,d^\times d \\
    &= \int_{(F_v^\times)^2} |a|^{s+z-3}|d|^{z-2} W_v(m(a,d))\int_{ F_v } f_v( w \, n(b ) , s  , z) \psi_v(bd) db \, d^\times a \,d^\times d. 
\end{align*} 
Note that $W_v(m(a,d))$ is zero unless the $v$-adic valuations of $a$ and $d$ are non-negative. Moreover, if $b \ne 0$, the inner 2-by-2 block of $w \, n(b ) w^{-1} $ can be written as \[\left(\begin{smallmatrix} 1 & -b^{-1}\\ & 1\end{smallmatrix} \right) \left(\begin{smallmatrix} b^{-1} &  \\& b \end{smallmatrix} \right) \left(\begin{smallmatrix}   & 1  \\ -1& b^{-1} \end{smallmatrix} \right).\]
In particular, when $|b| > 1$, it is right $\GL_2(\mathcal{O}_v)$-equivalent to $\left(\begin{smallmatrix} 1 & -b^{-1}\\ & 1\end{smallmatrix} \right) \left(\begin{smallmatrix} b^{-1} &  \\& b \end{smallmatrix} \right)$. Therefore \[f_v( w \, n(b ) , s  , z) = \begin{cases}
    |b|^{z-s-1} & \text{ if } |b| > 1 \\ 1 & \text{ if } |b| \leq 1,
\end{cases} \] and the inner integral becomes 

\begin{align*}
    1 + \int_{ |b| > 1 } |b|^{z-s-1} \psi_v(bd) db &= 1 + \sum_{\ell =1}^\infty q_v^{\ell(z-s)} \int_{|b|=1} \psi_v(\varpi_v^{n-\ell} b) db, \end{align*}
    with $n$ being the $v$-adic valuation of $d$. We claim that this integral equals to \[ (1 - q_v^{(n+1)(z-s)})\frac{\zeta_v(s-z)}{\zeta_v(s-z+1)}.\]
    To verify our claim, we treat the cases of $n=0$ and $n\geq 1$ separately. We start by the case of $n=0$. Then this integral is
\begin{align}\label{unipCasen0}
    1 + \sum_{\ell =1}^\infty q_v^{\ell(z-s)} \int_{|b|=1} \psi_v(\varpi_v^{-\ell} b) db.\end{align}
As $\psi_v$ has conductor $\mathcal{O}_v$, we have 
    \[\int_{\mathfrak{p}_v^i} \psi_v(b) db = \begin{cases}{\rm vol}(\mathfrak{p}_v^i) & \text{ if } i\geq 0, \\  
        0 & \text{ otherwise.}
    \end{cases} \]
    Therefore, in the sum only the term for $\ell =1$ survives. As
    \begin{align}\label{eqForCond1unip}
        \int_{|b|=1} \psi_v( \varpi_v^{-1} b) db = \int_{\mathfrak{p}_v^{-1}} \psi_v(   b) db- q_v^{-1} \cdot \int_{\mathcal{O}_v} \psi_v( b) db = - q_v^{-1},
    \end{align}     
    we have that \[\text{\eqref{unipCasen0}} = 1 - q_v^{(z-s-1)}= \frac{1}{\zeta_v(s-z+1)}.\]
    Suppose now that $n \geq 1$. Then the inner integral
\begin{align*}
    1 + \sum_{\ell =1}^\infty q_v^{\ell(z-s)} \int_{|b|=1} \psi_v(\varpi_v^{n-\ell} b) db  = 1 + {\rm vol}(\mathcal{O}_v^\times)\sum_{\ell=1}^n q_v^{\ell(z-s)}  +  \sum_{\ell =n+1}^\infty q_v^{\ell(z-s)} \int_{|b|=1} \psi_v( \varpi_v^{n-\ell} b) db. \end{align*}
    Using again that $\psi_v$ has conductor $\mathcal{O}_v$, in the last sum only the term for $\ell =n+1$ survives and \eqref{eqForCond1unip} implies that
    \[ \sum_{\ell =n+1}^\infty q_v^{\ell(z-s)} \int_{|b|=1} \psi_v( \varpi_v^{-1} b) db= -q_v^{-1}  q_v^{(n+1)(z-s)} . \]
    Hence the integral is
    \begin{align*}
    &= 1 + (1 - q_v^{-1}) \sum_{\ell =1}^{n} q_v^{\ell(z-s)} -   q_v^{-1}  q_v^{(n+1)(z-s)}  \\ &= 1 + (1 - q_v^{-1})q_v^{z-s} \frac{1 - q_v^{n(z-s)}}{1-q_v^{z-s}} - q_v^{-1}    q_v^{(n+1)(z-s)} \\ 
    &= (1-q_v^{z-s-1})\frac{1 - q_v^{(n+1)(z-s)}}{1-q_v^{z-s}}\\ 
    &= (1 - q_v^{(n+1)(z-s)})\frac{\zeta_v(s-z)}{\zeta_v(s-z+1)}.
\end{align*}
This proves our claim, reducing the local integral $I(\varphi_v,f_v, s  , z )$ to 
\begin{equation}\label{localform}
   \frac{\zeta_v(s-z)}{\zeta_v(s-z+1)}\sum_{m,n \geq 0} q_v^{-m(s+z-3)-n(z-2)} W_v(m(\varpi_v^m,\varpi_v^n))(1 - q_v^{(n+1)(z-s)}).\end{equation}
Note that we have used that $W_v(m(\varpi_v^m,\varpi_v^n))$ is zero unless $m,n \geq 0$. In what follows, we use the Casselman--Shalika formula for $W_v$ to calculate \eqref{localform}. We conveniently split the calculation into two cases depending on whether $v$ is split or inert in $E$.  
 
\subsection{Unramified computation at split primes}\label{unrcompsplit}
Let $v \not \in \Sigma$ be a place which splits over $E$. As $\Pi_v$ has trivial central character, we can assume that it is the unique unramified subquotient of $\mathrm{Ind}_{B_{\GL_4}(F_v)}^{\GL_4(F_v)} \chi$ with $\chi$ an unramified character of $T_{\GL_4}(F_v)$ such that $\chi_1\chi_2\chi_3\chi_4=1$. As $\Pi_v$ has trivial central character, we can write \[W_v(m(\varpi_v^m,\varpi_v^n)) = W_v({\rm diag}(\varpi_v^{2m+n},\varpi_v^{m+n},\varpi_v^m,1)) \] Then, the Casselman--Shalika formula gives that 
\begin{align*}
    W_v({\rm diag}(\varpi_v^{2m+n},\varpi_v^{m+n},\varpi_v^m,1)) &=  \delta_{B_{\GL_4}}^{1/2}({\rm diag}(\varpi_v^{2m+n},\varpi_v^{m+n},\varpi_v^m,1)) {\rm Tr}(g_{\Pi_v}| V_{(m,n,m)}) \\ 
    &= q_v^{-3m-2n} {\rm Tr}(g_{\Pi_v}| V_{(m,n,m)}) 
\end{align*}
where $V_{(m,n,m)}$ is the irreducible representation of ${\rm SL}_4(\C)$ with highest weight $m \omega_1 + n \omega_2 + m \omega_3$. Here $\omega_1$,$\omega_2$, and $\omega_3$ are the fundamental weights of ${\rm SL}_4(\C)$ and are the highest weights of the standard representation, the exterior square representation, and the exterior cube representation respectively.

Hence \eqref{localform} looks like 
\[ \frac{\zeta_v(s-z)}{\zeta_v(s-z+1)}\sum_{m,n \geq 0} q_v^{-m(s+z)-nz} \cdot   {\rm Tr}(g_{\Pi_v}| V_{(m,n,m)}) (1 - q_v^{(n+1)(z-s)}). \]

If we let $X:= q_v^{ -s}$ and $Y:= q_v^{ -z}$, it reads as 
\begin{align*}
\text{ \eqref{localform} }&=\frac{1}{\zeta_v(s-z+1)}\sum_{m,n \geq 0} {\rm Tr}(g_{\Pi_v}| V_{(m,n,m)})  X^m Y^{m+n} \frac{1 - X^{n+1} Y^{-n-1}}{1 - X Y^{-1}} 
\\&= \frac{1}{\zeta_v(s-z+1)}\sum_{m,n \geq 0} {\rm Tr}(g_{\Pi_v}| V_{(m,n,m)})  X^m Y^{m+n} \sum_{\alpha = 0}^n X^\alpha Y^{-\alpha} \\ &= \frac{1}{\zeta_v(s-z+1)}\sum_{m,\alpha, \beta \geq 0} {\rm Tr}(g_{\Pi_v}| V_{(m,\alpha+\beta,m)})  X^{m+\alpha} Y^{m+\beta},
\end{align*}
where we did the change of variables writing $n= \alpha + \beta$. 
We manipulate \eqref{localform} further. Let $\boldsymbol{\zeta}_v( Z)$ denote $\sum_{k \geq 0 } Z^k$. Using Cauchy product formula, the product between \eqref{localform} and $\boldsymbol{\zeta}_v(XY)$ is
\begin{align*}
   \text{\eqref{localform}} \cdot \boldsymbol{\zeta}_v(XY) &=  \frac{1}{\zeta_v(s-z+1)} \left[\sum_{m \geq 0} \left( \sum_{\alpha, \beta \geq 0} {\rm Tr}(g_{\Pi_v}| V_{(m,\alpha+\beta,m)})  X^{\alpha} Y^{\beta} \right)X^{m}Y^{m}\right] \cdot \sum_{k \geq 0 } X^kY^k \\
   &=  \frac{1}{\zeta_v(s-z+1)} \sum_{m \geq 0} \sum_{a =0}^m \left( \sum_{\alpha, \beta \geq 0} {\rm Tr}(g_{\Pi_v}| V_{(a,\alpha+\beta,a)})  X^{\alpha} Y^{\beta} \right) X^{m}Y^{m} \\
 &=\frac{1}{\zeta_v(s-z+1)}\sum_{m,\alpha, \beta \geq 0} X^{m+\alpha} Y^{m+\beta} \sum_{a = 0}^m{\rm Tr}(g_{\Pi_v}| V_{(a,\alpha+\beta,a)}) \\
    &=\frac{1}{\zeta_v(s-z+1)}\sum_{a,b,\alpha, \beta \geq 0}{\rm Tr}(g_{\Pi_v}| V_{(a,\alpha+\beta,a)}) X^{a+b+\alpha} Y^{a+b+\beta}, 
    \end{align*} 
where in the fourth equality we have made a change of variable writing $m = a +b$. We now further manipulate the sum by multiplying $\text{\eqref{localform}} \cdot \boldsymbol{\zeta}_v(XY)$ with  $\boldsymbol{\zeta}_v( X^2)\boldsymbol{\zeta}_v( Y^2)$. Firstly, write $\boldsymbol{\zeta}_v( X^2) = \sum_{k \geq 0} N_k X^k$, with \[ N_k = \begin{cases} 1 & \text{ if } k \text{ is even,} \\ 0 & \text{ otherwise.} 
\end{cases}\]
Then the Cauchy product formula gives that $\text{\eqref{localform}} \cdot \boldsymbol{\zeta}_v(XY) \cdot \boldsymbol{\zeta}_v( X^2)$ equals 
\begin{align*}
     &= \frac{1}{\zeta_v(s-z+1)}\left [\sum_{\alpha \geq 0} \left ( \sum_{a,b, \beta \geq 0}{\rm Tr}(g_{\Pi_v}| V_{(a,\alpha+\beta,a)}) X^{a+b} Y^{a+b+\beta} \right) X^\alpha \right] \cdot  \sum_{k \geq 0} N_k X^k \\ 
     &= \frac{1}{\zeta_v(s-z+1)} \sum_{\alpha \geq 0} \sum_{k=0}^\alpha \left ( \sum_{a,b, \beta \geq 0} N_k \cdot {\rm Tr}(g_{\Pi_v}| V_{(a,\alpha-k+\beta,a)}) X^{a+b} Y^{a+b+\beta} \right) X^\alpha \\ 
      &= \frac{1}{\zeta_v(s-z+1)}  \sum_{a,b, c, k, \beta \geq 0} N_k \cdot {\rm Tr}(g_{\Pi_v}| V_{(a,c+\beta,a)}) X^{a+b + c + k} Y^{a+b+\beta} \\
      &= \frac{1}{\zeta_v(s-z+1)}  \sum_{a,b, c, d, \beta \geq 0} {\rm Tr}(g_{\Pi_v}| V_{(a,c+\beta,a)}) X^{a+b + c + 2d} Y^{a+b+\beta},
\end{align*}
where in the third equality we have made the change of variables by writing $\alpha = c + k$, and, since $N_k = 1$ for $k$ even and $N_k = 0$ otherwise, in the fourth one we made the change of variables $k=2d$. Proceeding in an analogous way for the multiplication by $\boldsymbol{\zeta}_v( Y^2)$, one gets that 
\[\text{\eqref{localform}} \cdot \boldsymbol{\zeta}_v(XY)  \cdot \boldsymbol{\zeta}_v( X^2) \cdot \boldsymbol{\zeta}_v( Y^2) = \sum_{a,b,c, d,e,f \geq 0} {\rm Tr}(g_{\Pi_v}| V_{(a,c+e,a)}) X^{a+b+c+2d} Y^{a+b+e+2f}, \]
or, equivalently, that
\begin{align}\label{eq:optimalformsplit}
\eqref{localform} = \frac{1}{\zeta_v(2z)\zeta_v(2s)\zeta_v(s+z)\zeta_v(s-z+1)}\sum_{a,b,c, d,e,f \geq 0} {\rm Tr}(g_{\Pi_v}| V_{(a,c+e,a)}) X^{a+b+c+2d} Y^{a+b+e+2f}.
\end{align}

Let $\boldsymbol{L}(Z, \Pi_v, \wedge^2)$ denote $\sum_{k \geq 0 } {\rm Tr}(g_{\Pi_v}|{\rm Sym}^k( V_{(0,1,0)}))Z^k$. Our goal is to compare the right hand side of \eqref{eq:optimalformsplit} with the product
\[\boldsymbol{L}(X,\Pi_v,\wedge^2)\boldsymbol{L}(Y,\Pi_v,\wedge^2).\]

We first need two lemmas.

\begin{lemma}\label{SymmetricFactSL4}
    For any $k \geq 0$, we have the following isomorphism
    \[\mathrm{Sym}^k\left(V_{(0,1,0)}\right) = \bigoplus_{i = 0}^{\lfloor k/2 \rfloor} V_{(0,k-2i,0)}.\]
\end{lemma}
\begin{proof}
For all $k \geq 2$, there is a surjection 
\[\mathrm{Sym}^k\left(V_{(0,1,0)}\right) \longrightarrow \mathrm{Sym}^{k-2}\left(V_{(0,1,0)}\right)\]
with kernel $V_{(0,k,0)}$ (\textit{cf.} \cite[Exercise 15.12]{FultonHarris}). The result follows from applying the surjection recursively.
\end{proof}

\begin{lemma}\label{LRrulesplit}
 For any $n,m \geq 0$, we have \[V_{(0,n,0)} \otimes V_{(0,m,0)} = \bigoplus_{\substack{m_1,m_2,m_3 \geq 0 \\ \sum m_i =m \\ m_2+m_3 \leq n}}\!\!\!\! V_{(m_2,n+m_1-m_2-m_3,m_2)}. \]   
\end{lemma}
\begin{proof}
Following the discussion in \cite[\S 25.3]{FultonHarris}, we apply the Littlewood--Richardson rule as in \cite[(A.8)]{FultonHarris} (\textit{cf}. \cite[I.9]{macdonald} for a proof of the formula). If we let $D_t$ denote the Young diagram with rows of size $(t,t)$, we have 
\[V_{(0,n,0)} \otimes V_{(0,m,0)} = \oplus_{\lambda} N_{n,m,\lambda} V_\lambda,\]
where $N_{n,m,\lambda}$ is the number of strict expansions of $D_n$ by $D_m$ of shape the Young diagram of $V_\lambda$. In our case $N_{n,m,\lambda}$ is either zero or one; hence we are left to determine the set of strict expansions of $D_n$ by $D_m$.  Recall that an expansion of $D_n$ by $D_m$ is a Young diagram  obtained by first adding $m$ boxes and putting the integer $1$ in each of these so that no two 1's are in the same column and then adding other $m$ boxes with a 2 so that no two 2's are in the same column. Such expansion is called strict if when reading from right to left, starting with the top row and working down, for any $1 \leq t \leq 2m$ the number of 1's up to position $t$ is greater or equal to the number of 2's up to position $t$. It follows that the only way we can add the first $m$ boxes is by getting the Young diagram with rows of size $(n+\ell_1,n,\ell_2)$ such that $\ell_1 + \ell_2 = m$ and $\ell_2 \leq n$. If we let $(m_1,m_2,m_3)$ be any partition of $m$, the $m$ boxes labeled by $2$ can be added by appending $m_1$ boxes to the second row, $m_2$ to the third, and $m_3$ to the fourth so that $\ell_1 \geq m_1 + m_2$, $\ell_2 \geq m_3$, $\ell_2 + m_2 \leq n$. This implies that $\ell_1 = m_1 + m_2$ and $\ell_2 = m_3$,  hence any strict expansion of $D_n$ by $D_m$ is determined by choosing a partition $(m_1,m_2,m_3)$ of $m$ such that $m_2 + m_3 \leq n$. The irreducible representation of $\mathrm{SL}_4(\C)$ corresponding to the strict expansion of $D_n$ by $D_m$ associated to $(m_1,m_2,m_3)$ is $V_{(m_2,n+m_1-m_2-m_3,m_2)}$. This proves the formula.
\end{proof}

\begin{theorem}\label{SplitTheorem}
We have that \[ I(\varphi_v,f_v, s  , z ) = \frac{L(s,\Pi_v,\wedge^2)L(z,\Pi_v,\wedge^2)}{\zeta_v(2z)\zeta_v(2s)\zeta_v(s+z)\zeta_v(s-z+1)}. \]
In particular, if we normalize $f_v^*(g,s,z):= \zeta_v(2z)\zeta_v(2s)\zeta_v(s+z)\zeta_v(s-z+1) f_v(g,s,z)$, \[ I(\varphi_v,f_v^*, s  , z ) =L(s,\Pi_v,\wedge^2)L(z,\Pi_v,\wedge^2). \]
\end{theorem}
    
\begin{proof}
    By Lemma \ref{SymmetricFactSL4}, the coefficient of $X^m Y^n$ of $\boldsymbol{L}(X,\Pi_v,\wedge^2)\boldsymbol{L}(Y,\Pi_v,\wedge^2)$ is given by 
    \[\sum_{i=0}^{\lfloor \frac{m}{2}\rfloor} \sum_{j=0}^{\lfloor \frac{n}{2}\rfloor} {\rm Tr}(g_{\Pi_v}|V_{(0,m-2i,0)})\mathrm{Tr}(g_{\Pi_v}|V_{(0,n-2j,0)}).  \]
By Lemma \ref{LRrulesplit}, we thus have that
\[\boldsymbol{L}(X,\Pi_v,\wedge^2)\boldsymbol{L}(Y,\Pi_v,\wedge^2) = \sum_{m,n \geq 0} X^mY^n\sum_{i=0}^{\lfloor \frac{m}{2}\rfloor} \sum_{j=0}^{\lfloor \frac{n}{2}\rfloor}  \sum_{\substack{t_1,t_2,t_3 \geq 0 \\ \sum t_i =n-2j \\ t_2+t_3 \leq m - 2i}} {\rm Tr}(g_{\Pi_v}|V_{(t_2,m-2i - t_2 - t_3 +t_1,t_2)}).  \]
Do a change of variables by  $m = r_1 + t_2 + t_3 + 2i$ and $n = t_1 +t_2 +t_3 + 2j $ to get that
\[\boldsymbol{L}(X,\Pi_v,\wedge^2)\boldsymbol{L}(Y,\Pi_v,\wedge^2) = \sum_{r_1,t_1,t_2,t_3,i,j \geq 0} {\rm Tr}(g_{\Pi_v}|V_{(t_2,r_1 + t_1,t_2)}) X^{r_1 + t_2 + t_3 + 2i}Y^{t_1 +t_2 +t_3 + 2j}.  \]
This equals the infinite sum in the right hand side of \eqref{eq:optimalformsplit}. The theorem follows. 
\end{proof}

\subsection{Unramified computation at inert primes}

Let $v \not \in \Sigma$ be a place   which is inert over $E$ and let us assume that $\Pi_v$ is the unique unramified subquotient of $\mathrm{Ind}_{B_{\G}(F_v)}^{\G(F_v)}\chi$ with $\chi$ an unramified character of $T_{\G}(F_v)$. Denote by $\pi_v$ the unique unramified subquotient of $\mathrm{Ind}_{B_{\H}(F_v)}^{\H(F_v)}\chi_{|_{T_{\H}}}$ and by $g_{\pi_v}$ its Frobenius conjugacy class.  As $\Pi_v$ has trivial central character, we can write \[W_v(m(\varpi_v^m,\varpi_v^n)) = W_v({\rm diag}(\varpi_v^{2m+n},\varpi_v^{m+n},\varpi_v^m,1)). \] As the relative root system of $\G_{/F_v}$ is that of $\H_{/F_v}$, the Casselman--Shalika formula for the Whittaker coefficient $W_v$ of $\Pi_v$ gives that (\textit{cf}. \cite[Proposition 5.2.1]{GanHundley})
\begin{align*}
    W_v({\rm diag}(\varpi_v^{2m+n},\varpi_v^{m+n},\varpi_v^m,1)) &=  \delta_{B_{\G}}^{1/2}({\rm diag}(\varpi_v^{2m+n},\varpi_v^{m+n},\varpi_v^m,1)) {\rm Tr}(g_{\pi_v}| W_{(m,n)}) \\ 
    &= q_v^{-3m-2n} {\rm Tr}(g_{\pi_v}| W_{(m,n)}) 
\end{align*}
where $W_{(m,n)}$ is the irreducible representation of ${\rm Spin}_5(\C)$ with highest weight $m \omega_1 + n \omega_2$. Here $\omega_1$ and $\omega_2$ are the fundamental weights of ${\rm Spin}_5(\C)$ and correspond to the five-dimensional standard orthogonal representation and to the four-dimensional spin representation, respectively. Hence, if we let $X:= q_v^{ -s}$ and  $Y:= q_v^{ -z}$, \eqref{localform} is 
\[\frac{1}{\zeta_v(s-z+1)}\sum_{m,n \geq 0} {\rm Tr}(g_{\pi_v}| W_{(m,n)})  X^m Y^{n+m} \frac{1 - X^{n+1} Y^{-n-1}}{1 - X Y^{-1}} . \]
After multiplying and dividing by $\boldsymbol{\zeta}(XY)$, we have
\begin{align}\label{localforminert}
    \text{\eqref{localform}}&=\frac{1}{\zeta_v(z+s)\zeta_v(s-z+1)}\sum_{a,b,\alpha, \beta \geq 0}{\rm Tr}(g_{\pi_v}| W_{(a,\alpha+\beta)}) X^{a+b+\alpha} Y^{a+b+\beta}, 
    \end{align}
where we have written $m = a+b$ and $n=\alpha + \beta$. Let $\boldsymbol{L}(Z, \pi_v, {\rm Spin})$ denote $\sum_{k \geq 0 } {\rm Tr}(g_{\pi_v}|{\rm Sym}^k( W_{(0,1)}))Z^k$. We now compare the right hand side of \eqref{localforminert} with the product
\[\boldsymbol{L}(X, \pi_v, {\rm Spin})\boldsymbol{L}(Y, \pi_v, {\rm Spin}),\]
which we calculate as follows. Making use of the exceptional isomorphism $\mathrm{Spin}_{5}(\C)\xrightarrow{\simeq}\mathrm{Sp}_4(\C)$, the explicit branching law for
    \[\mathrm{Spin}_{5}(\C)\xrightarrow{\simeq}\mathrm{Sp}_4(\C)\hookrightarrow \SL_4(\C)\]
reads as (\textit{cf}. \cite[p. 739]{BFG})
\begin{align}\label{branching}
    V_{(u,v,w)} = \bigoplus_{s=0}^v \bigoplus_{t=0}^{{\rm min}(w,u)} W_{(s+t,u+w-2t)},
\end{align} 
where $V_{(u,v,w)}$ denotes the irreducible representation of $\SL_4(\C)$ as in \S \ref{unrcompsplit}. As a special case, we see that ${\rm Sym}^n(V_{(1,0,0)})= V_{(n,0,0)} = W_{(0,n)} = {\rm Sym}^n(W_{(0,1)})$ for every $n \geq 0$.

\begin{lemma}\label{LRruleinert}
For any $n,m \geq 0$, we have \[W_{(0,n)} \otimes W_{(0,m)} = \bigoplus_{\substack{m_1,m_2 \geq 0 \\ m_1 +m_2 =m \\ m_2 \leq n}} \bigoplus_{s=0}^{m_2} W_{(s,n+m_1-m_2)}. \] 
\end{lemma}

\begin{proof}
We start by computing the tensor product of $V_{(n,0,0)}$ with $V_{(m,0,0)}$ by the Littlewood-Richardson rule. Let $Y_t$ denote the Young diagram with a row of size $t$. Each strict expansion of $Y_n$ by $Y_m$ correspond to a pair of positive integers $(m_1,m_2)$ such that $m_1 + m_2 = m $ and $m_2 \leq n$ and it gives the Young diagram of $V_{(m_1+n-m_2,m_2,0)}$. Hence, we have 
\[V_{(n,0,0)} \otimes V_{(m,0,0)} = \bigoplus_{\substack{m_1,m_2 \geq 0 \\ m_1 +m_2 =m \\ m_2 \leq n}}\!\!\!\! V_{(n+m_1-m_2,m_2,0)}. \]
This and \eqref{branching} give the desired formula. 
\end{proof}

\begin{theorem}\label{inertTheorem}
We have that \[ I(\varphi_v,f_v, s  , z ) = \frac{L(s,\Pi_v,\wedge^2)L(z,\Pi_v,\wedge^2)}{\zeta_v(2z)\zeta_v(2s)\zeta_v(s+z)\zeta_v(s-z+1)}. \]
In particular, if we normalize $f_v^*(g,s,z):= \zeta_v(2z)\zeta_v(2s)\zeta_v(s+z)\zeta_v(s-z+1) f_v(g,s,z)$, \[ I(\varphi_v,f_v^*, s  , z ) =L(s,\Pi_v,\wedge^2)L(z,\Pi_v,\wedge^2). \]
\end{theorem}
    
\begin{proof}
    By Lemma \ref{LRruleinert}, the coefficient of $X^m Y^n$ of $\boldsymbol{L}(X,\pi_v,{\rm Spin})\boldsymbol{L}(Y,\pi_v,{\rm Spin})$ is given by 
    \[\sum_{\substack{m_1,m_2 \geq 0 \\ m_1 + m_2 = m \\ m_2 \leq n}} \sum_{s=0}^{m_2} {\rm Tr}(g_{\pi_v}|W_{(s, n + m_1 - m_2)}).  \]

Write $m_2 = s + r$, $n = t + s +r$, and $m= m_1 +s +r$, then this implies that \[\boldsymbol{L}(X,\pi_v,{\rm Spin})\boldsymbol{L}(Y,\pi_v,{\rm Spin}) = \sum_{m_1, s, r ,t \geq 0 } {\rm Tr}(g_{\pi_v}|W_{(s, m_1 +t)}) X^{s+r+m_1}Y^{s+r+t}. \] 
This is exactly the infinite sum in the right hand side of \eqref{localforminert}.
Then the theorem follows from \eqref{extTOspin}, which says that, for $s' \in \{s,z\}$  \[ L(s',\Pi_v,\wedge^2) = \zeta_v(2s') L(s',\pi_v,{\rm Spin}).\]
\end{proof}

\subsection{Ramified computations}

Let $\Sigma_\infty \subseteq \Sigma$ denote the set of archimedean places of $F$.

\begin{lemma}\label{LemmaonRamified}\leavevmode
\begin{enumerate}
    \item For any place $v$, there exists $\epsilon > 0$ such that  $I(\varphi_v,f_v,s,z)$ converges absolutely for ${\rm Re}(z)>1-\epsilon$ and ${\rm Re}(s)>{\rm Re}(z) -4$. Moreover, $\prod_{v \in \Sigma} I(\varphi_v,f_v,s,z)$
    has meromorphic continuation in $s$ and $z$.
    \item For any finite place $v$, there exist $\varphi_v$ and $f_v(g,s,z)$ such that $I(\varphi_v,f_v,s,z)$ is a non-zero constant independent of $s$ and $z$. 
    \item For any place $v \in \Sigma_\infty$, given $s_0, z_0 \in \C$, one can choose $\varphi_v$ and $f_v(g,s,z)$ such that $I(\varphi_v,f_v,s,z)$ is non-zero at $(s_0,z_0)$.
\end{enumerate}
\end{lemma}

\begin{proof}
We will reduce our first claim to the study in \cite[\S 3.1]{FurusawaMorimoto} of the local integrals \cite[(3.1)]{FurusawaMorimoto}. To do so, we choose the section defining the Borel Eisenstein series as follows. Let $v$ be any place of $F$. Let $W_4$ be the standard representation of $\H(F_v)$, with basis $\{e_1,e_2,f_2,f_1\}$ and let $W_{1,1}$ be the five-dimensional representation of $\H(F_v)$ of highest weight $(1,1)$. It sits inside the exterior square of $W_4$ and in particular contains the vector $f_2 \wedge f_1$.  For a Schwartz--Bruhat function $\Phi_{P}$ on $W_{1,1}$, define the function on $\H(F_v)$ given by
\[ f_{P_\H} (\Phi_{P}, g, s-z+1) := | \mu(g) |^{ s-z+1} \int_{F^\times_v} \Phi_{P}( t f_2 \wedge f_1 g) |t|^{2(s-z+1)} dt \in I_{P_\H(F_v)}^{\H(F_v)}( \delta_{P_\H}^{\frac{1}{3}(s-z+1)}). \]
Similarly, for a Schwartz--Bruhat function $\Phi_{Q}$ on $M_{1\times 4}(F_v)$, where we let $\H(F_v)$ acts on it by right multiplication, define the function on $\H(F_v)$ given by
\[ f_{Q_\H} (\Phi_{Q}, g, z) := | \mu(g) |^{z} \int_{F^\times_v} \Phi_{Q}( t f_1 g) |t|^{2z} dt \in I_{Q_\H(F_v)}^{\H(F_v)}( \delta_{Q_\H}^{\frac{1}{2}z}). \]
Here $I_{P_\H(F_v)}^{\H(F_v)}( \delta_{P_\H}^{\frac{1}{3}(s-z+1)})$ and $ I_{Q_\H(F_v)}^{\H(F_v)}( \delta_{Q_\H}^{\frac{1}{2}z})$ are not normalized inductions. It's easy to see that $f_v(g, s, z):= f_{P_\H} (\Phi_{P}, g, s-z+1)f_{Q_\H} (\Phi_{Q}, g, z) \in I_{B_\H(F_v)}(s,z)$. Call $\tilde{P}$ the subgroup of $P_{\H}$ given by $\GL_2 N_{P_\H}$, with $\GL_2 \hookrightarrow M_{P_\H}$ via $h \mapsto h^*:= \left(\begin{smallmatrix}h& \\ & {\rm det}(h) J_2 {}^th^{-1} J_2\end{smallmatrix}\right)$. With this choice of section and collapsing the sum with respect to $\tilde{P}_w(F_v):=w^{-1} \tilde{P}(F_v) w$, the local integral  
\begin{align*}
I(\varphi_v, f_v, s, z) &= \int_{\tilde{P}_w(F_v)\setminus \H(F_v)} \int_{N_{Q_\H}(F_v)Z_\H(F_v)\setminus\tilde{P}_w(F_v)} \!\!\!\!\!\!f_v(w p_w h , s  , z)   W_{\varphi_v}(p_w h) d p_w dh \\ 
&= \int_{\cdots} \int_{N_{Q_\H}(F_v)Z_\H(F_v)\setminus \tilde{P}_w(F_v)} \!\!\!\!\!\!f_{P_\H} (\Phi_{P}, pwh, s-z+1)f_{Q_\H} (\Phi_{Q}, p_w h, z)W_{\varphi_v}(p_w h) dp_w dh
\\ 
&= \int_{\cdots} f_{P_\H} (\Phi_{P}, wh, s-z+1) \int_{N_{Q_\H}(F_v)Z_\H(F_v)\setminus\tilde{P}_w(F_v)} \!\!\!\!\!\!f_{Q_\H} (\Phi_{Q}, p_w h, z)W_{\varphi_v}(p_w h) dp_w dh,
\end{align*}
where we have written $p_w = w^{-1} p w$ for $p \in \tilde{P}(F_v)$, we have used that $w$ acts trivially on $f_{Q_\H}$  and in the last equality that the modulus character of the Siegel parabolic is trivial on $\tilde{P}(F_v)$. We now look at the inner integral. The domain $N_{Q_\H}(F_v)Z_\H(F_v)\setminus  \tilde{P}_w(F_v)$ is isomorphic to $\left(Z_{\GL_2}(F_v)U_{\GL_2}(F_v)\setminus \GL_2(F_v)\right) \cdot F_v$, where $\GL_2(F_v)$ and $F_v$ embed into $\H(F_v)$ as 

\[ \iota: \left(\begin{smallmatrix}a& b \\c & d\end{smallmatrix}\right) \mapsto \left(\begin{smallmatrix}a & & b & \\ & a& &b\\ c& & d& \\& c & & d \end{smallmatrix}\right),\, u:\alpha \mapsto \left(\begin{smallmatrix}1 & &  & \\ &1 & & \\ & -\alpha &1 & \\&  & & 1 \end{smallmatrix}\right).  \]
This let us write the inner integral as 
\begin{align}\label{innerlocalintegralram}
\int_{Z_{\GL_2}(F_v)U_{\GL_2}(F_v)\setminus \GL_2(F_v)} \int_{F_v} f_{Q_\H} (\Phi_{Q}, \iota(g)h, z)W_{\varphi_v}(u(\alpha) \iota(g) h) d \alpha \, d g,
\end{align} where we used that $ u(\alpha)$ is in the Levi of the Klingen parabolic and its image via $\delta_{Q_\H}$ is trivial, so that it acts trivially on $f_{Q_\H}$. Note that for certain $\Phi_Q$,   \eqref{innerlocalintegralram} coincides with the local zeta integral of \cite[(3.1)]{FurusawaMorimoto}. Let $P_w := w^{-1} P_{\H} w$; by the Iwasawa decomposition, there is a compact subgroup $K_w$ of $\H(F_v)$ such that $P_w(F_v) K_w = \H(F_v)$. Moreover, the quotient $\GL_1 \simeq \tilde{P}_w\backslash P_w$ via the map 
\[ \iota': r \mapsto \left(\begin{smallmatrix}1 & &  & \\ & r & & \\ &   &1 & \\&  & & r \end{smallmatrix}\right). \]
Note that $\iota'(r)\iota(g) =\iota(g) \iota'(r)$. Thus, using this and doing the change of variables $u(\alpha) \mapsto  \iota'(r)^{-1} u(\alpha) \iota'(r)$, we write $I(\varphi_v, f_v, s, z)$ as
\begin{align*}
   \int_{K_w}\!\!\!\! f_{P_\H} (\Phi_{P}, w k, s-z+1) \int_{Z_{2}U_{2}\setminus \GL_2} \int_{F_v}\!\!\!\! f_{Q_\H} (\Phi_{Q}, \iota(g)k, z) \int_{F_v^\times} \!\!\!\!\!\! |r|^{-s-2}  W_{\varphi_v}(\iota'(r) u(\alpha) \iota(g) k) d r \, d \alpha \, d g\, d k. 
\end{align*} 
We now apply Dixmier--Malliavin theorem to the maximal unipotent subgroup $N$ of the Levi of $Q_\H$ to write $\varphi_v = \sum_j \xi_j * \varphi_{v,j}$ where  $\varphi_{v,j}$ are smooth vectors of $\Pi_v$ and $\xi_j \in C^\infty_c(N(F_v))$. Using the defining property of the Whittaker functional, we have 
\[ \int_{F_v^\times} \!\!\!\!\!\! |r|^{-s-2}  W_{\varphi_{v}}(\iota'(r) u(\alpha) \iota(g) k) d r = \sum_j  \int_{F_v^\times} \!\!\!\!\!\! |r|^{-s-2} \hat{\xi}_j(r^{-1}) W_{\varphi_{v,j}}(\iota'(r) u(\alpha) \iota(g) k) d r,\]
where 
\[\hat{\xi}_j(r) = \int_{N(F_v)} \xi_j\left(\left( \begin{smallmatrix}
    1 & n \\ & 1
\end{smallmatrix} \right)\right) \psi_v(r n) dn.\]
Note that $\hat{\xi}_j$ are compactly supported functions on $F_v^\times$. Our first claim now follows from \cite[Proposition 3.2]{FurusawaMorimoto}. Indeed, using the Iwasawa decomposition for $\GL_2(F_v) = U_2 T K_2$ where $T = \left \{ \left(\begin{smallmatrix}t& \\ & 1 \end{smallmatrix}\right) \right \}$, and \cite[Proposition 3.1]{FurusawaMorimoto}, a direct computation shows that each integral is majorated by 
\[ \left (\int_{F_v^\times} | \hat{\xi}_j(r^{-1})| |r|^{-s+z-4} dr \right ) \cdot \left( \int_{K_w}  \int_{K_{2}} \int_{F_v^\times}\int_{F_v}\!\!\!\! |t|^{z-1}
|W_{\varphi_{v,j}}(u(\alpha) \iota(t) k_2 k)|  d \alpha \, d t \, d k_2 \, dk \right). \]
The first converges for ${\rm Re}(s - z + 4) >0$, while by \cite[Proposition 3.2]{FurusawaMorimoto} there exists $\epsilon > 0$ such that the second converges for ${\rm Re}(z)>1-\epsilon$. This proves the first statement of (1). The meromorphic continuation of $\prod_{v\in \Sigma} I_v(\varphi_v,f_v,s,z)$ follows from our unramified computations as $I(\varphi,f,s,z)$ and the $L$-function $L^\Sigma(s,\Pi,\wedge^2)$ have meromorphic continuation (\textit{cf}. \cite[\S 7]{KimExterior}). The statements (2) and (3) follow from applying the Dixmier--Malliavin theorem as in \cite[Lemma 9.1]{Gan-SavinExceptionalSW}. 
\end{proof}

\section{Residues and the degree 5 \texorpdfstring{$L$}{L}-function of \texorpdfstring{$\GSp_4$}{GSp4}}

\subsection{A first-term identity between Eisenstein series for parabolics of $\GSp_4$}

\subsubsection{Siegel Eisenstein series on $\GSp_4$}\label{SiegelEisSeries}
We define the Siegel Eisenstein series associated to $\H$ and recollect some of their properties (for instance, see \cite[\S 5]{File13}). Recall from \S \ref{parabolics} the Siegel parabolic $P_\H$ of $\H$ with Levi decomposition $M_{P_\H} N_{P_\H}$ and modulus character $\delta_{P_\H}$. Let $s$ be a complex variable and let $I_{P_\H}(s)$ be the degenerate principal series representation of $\H(\A)$ consisting of smooth functions $f$ on $\H(\A)$ such that \[ f(n m g,s) = \delta_{P_\H}^{\tfrac{1}{3}(s + 1)}(m) f(g,s),\; \forall n \in N_{P_\H}(\A), \forall m \in M_{P_\H}(\A).  \]
Given a standard holomorphic section  $f \in I_{P_\H}(s)$, we define the degenerate Eisenstein series \[{\rm Eis}_{P_\H}(g, s, f) = \sum_{\gamma \in {P_\H}(F) \backslash \H(F)} f(\gamma g,s),\]
which is absolutely convergent in the half-plane ${\rm Re}(s) > 2$ and admits analytic continuation to a meromorphic function on $\C$ which satisfies a functional equation 
\begin{equation*}
    {\rm Eis}_{P_\H}(g,s, f) = {\rm Eis}_{P_\H}(g, 1-s, M_s f),
\end{equation*}
where $M_s$ is a certain intertwining operator. At a finite place $v$, let  $f_{v}^0$ be the function in $I_{P_\H(F_v)}(s)$ such that $f_{v}^0(k,s) = 1$ for all $k \in \H(\mathcal{O}_v)$. We assume that $f$ is a pure tensor and denote by $f_{v}$ its $v$-adic component, which equals $f_{v}^0$ almost everywhere. Let $S$ be the finite set of places containing the archimedean places and all $v$ such that  $f_{v} \ne f_{v}^0$. We then define the normalized Siegel Eisenstein series to be equal to 
\begin{equation}\label{normalizedSES}
    E^*_{P_\H}(g,s) := \zeta^S(s+1) \zeta^S(2s){\rm Eis}_{P_\H}(g,s, f).
\end{equation}
It follows from \cite[Theorem 1.1 and Theorem 4.12]{KudlaRallis} that $E^*_{P_\H}(g,s)$ has at most simple poles at the points $s = 1,2$. While $s=2$ is the end of the critical strip and so the residue of the Eisenstein series is the constant representation, i.e. there exists $f \in I_{P_\H}(s)$ such that ${\rm Res}_{s=2} E_{P_\H}^*(g,s)$ is constant, the residue of $E^*_{P_\H}(g,s)$ at $s=1$ is related to a value of a Klingen Eisenstein series by a result of Ikeda described in Proposition \ref{Ikeda}.

\subsubsection{The Klingen Eisenstein series and the first-term identity}

The Klingen parabolic $Q_\H = M_{Q_\H} N_{Q_\H}$ with modulus character $\delta_{Q_\H}$ was introduced in \S \ref{parabolics}. Let $I_{Q_\H}(s)$ be the degenerate principal series representation of $\H(\A)$ consisting of smooth functions $f'$ on $\H(\A)$ such that \[ f'(n m g,s) = \delta_{Q_\H}^{\tfrac{1}{2}(s + 1/2)}(m) f'(g,s),\; \forall n \in N_{Q_\H}(\A), \forall m \in M_{Q_\H}(\A).  \]
Given a smooth section  $f' \in I_{Q_\H}(s)$, we define the degenerate (Klingen) Eisenstein series \[   {\rm Eis}_{Q_\H}(g,s, f') =  {\sum_{\gamma \in {Q_\H}(F) \backslash \H(F)}} f'(\gamma g,s).\]
In the study of the poles of our integral we will employ the following result of Ikeda.

\begin{proposition}\label{Ikeda}
There exists a section $f_{1/2}' \in I_{Q_\H}(1/2)$ such that
\[{\rm Res}_{s =1 } E_{P_\H}^*(g,s) = \frac{{\rm Res}_{s=1}(\zeta^S_F(s))\zeta^S_F(3)}{2 \zeta^S_F(2)} {\rm Eis}_{Q_\H}( g , 1/2,f_{1/2}').  \]
\end{proposition}

\begin{proof}

This is basically \cite[Proposition 1.10]{Ikeda}. Notice that the result of \emph{loc.cit.} is for $\Sp_4$-Eisenstein series. In order to use it, we consider the following. By \cite[p. 11]{Harris-Kudla}, we have a factorization 
\[\H(\A) = \H(F)Z_{\H}(\A)\Sp_4(\A)\H(\A_f).\]
Then, given any Eisenstein series $E(g,s,f)$ defined over $\H(\A)$, we may write
\[E(g, s, f) = E(g_0 z g_1g_f, s, f) = E(g_1g_f, s, f) = E(g_1, s, r(g_f)f),\]
by the invariant properties of the Eisenstein series, and where $r(g_f)$ denotes the right multiplication action of $g_f$ on $f$.
In particular, the values of the Eisenstein series $E_{P_\H}(g,s,f)$ and  ${\rm Eis}_{Q_\H}(g,s,f')$ at $g$ equal to the values at $g_1$ of the Eisenstein series $E_{P_{\Sp_4}}(g_1,s, r(g_f)f)$, $E_{Q_{\Sp_4}}(g_1,s, r(g_f)f')$ for $\Sp_4$  associated respectively to the parabolics $P_{\Sp_4}=P_\H \cap \Sp_4$ and $Q_{\Sp_4}=Q_\H \cap \Sp_4$, and to the elements $r(g_f) f \in I_{P(\A)}^{\Sp_4(\A)} \nu_P^{s+1}$, $r(g_f) f' \in I_{Q(\A)}^{\Sp_4(\A)} \nu_Q^{2s+1}$, where the inductions are not normalized. Here we denoted 
\begin{align*}
    \nu_P:&\left(\begin{smallmatrix} g & \star \\ &  J_2 {}^tg^{-1} J_2 \end{smallmatrix} \right)  \mapsto \left|{\rm det}(g)\right|, \\ 
     \nu_Q:&\left(\begin{smallmatrix} a &  &  \\ & g & \\  & & a^{-1} \end{smallmatrix} \right)  \mapsto \left|a\right|.
\end{align*}
With respect to \cite{Ikeda}, the Eisenstein series $E_{P_{\Sp_4}}(g_1,s, r(g_f)f)$, resp. $E_{Q_{\Sp_4}}(g_1,s, r(g_f)f')$ have complex variable $s$ shifted by $-1/2$. We now apply \cite[Proposition 1.10]{Ikeda}, which says that, given $\omega$ the longest element of the quotient of Weyl groups $ W_{M_{P_\H}} \backslash W_{\H}$ and $M_\omega$ the corresponding intertwining operator, then
\[{\rm Res}_{s=1} E_{P_{\Sp_4}}(g_1,s, r(g_f)f) = \tfrac{1}{2} E_{Q_{\Sp_4}}(g_1, 1/2, {\rm Res}_{s=1} M_\omega( r(g_f)f)).\]
As $M_\omega$ is an intertwining map, the action of $g_f$ commutes with it, hence we have 
\[{\rm Res}_{s=1} E_{P_{\Sp_4}}(g_1,s, r(g_f)f) = \tfrac{1}{2} E_{Q_{\Sp_4}}(g_1, 1/2, {\rm Res}_{s=1} r(g_f)M_\omega( f)).\]
This implies that ${\rm Res}_{s=1}{\rm Eis}_{P_\H}(g,s,f)$ equals to  $\tfrac{1}{2}{\rm Eis}_{Q_\H}(g,s,{\rm Res}_{s=1} M_\omega(f))$, i.e. $f_{1/2}'$ can be taken to be ${\rm Res}_{s=1} M_\omega(f)$. After normalizing $f$, we thus get the desired result for the normalized Eisenstein series $E_{P_\H}^*(g,s)$ (see also \cite[Proposition 1.8]{Ikeda}).
\end{proof}

\subsubsection{The constant term of the Borel Eisenstein series}

The following informal discussion will not be used in the article, but we add it as a motivation for the normalization of the Borel Eisenstein series and the existence of its poles along the line $z=1$.

Let $W_{\H}$ denote the Weyl group of $\H$; also let $\Lambda_{s,z} := (2z-1)\beta_1+(s-z)\beta_2\in X^{*}(T_\H)$, where $\beta_1,\;\beta_2$ denote the fundamental weights of $\H$. Let $S$ be a finite set of places such that the defining datum of the Eisenstein series $E_{B_{\H}}(g,s,z)$ is unramified outside $S$. The analytic properties of $E_{B_{\H}}(g,s,z)$ are reflected in its constant term, that is equal to 
\[\sum_{\omega\in W_{\H}}M(\omega,\Lambda_{s,z})f(g,s,z),\]
where $M(\omega,\Lambda_{s,z})f(g,s,z) = \int_{(U_\H \cap \omega \overline{U}_\H \omega^{-1})(\A)}f(\omega^{-1}ug,s,z)du$, with $\overline{U}_\H$ the unipotent radical of the lower-triangular Borel, is the well known intertwining map. The Gindikin-Karpelevic formula shows that the poles of $M(\omega,\Lambda_{s,z})$ are exactly the ones of the following functions: \[\prod_{\substack{\alpha>0\\\omega\alpha<0}}\zeta^S_F(\left<\Lambda_{s,z},\alpha^{\vee}\right>),\] 
where $\alpha$ ranges through the positive roots for $\H$ such that $\omega \alpha$ is negative. 
Therefore, the only possible poles are located along the lines determined by $\left<\Lambda_{s,z},\alpha^{\vee}\right> = 1$. Concretely, these are $\{z = 1\}$, $\{s = 1\}$, $\{s-z = 1\}$, and $\{z+s=2\}$. In fact, the constant term is a combination of the following zeta functions:
\[\zeta^S_F(2z-1),\;\;\zeta^S_F(s-z),\;\;\zeta^S_F(s+z-1),\;\;\zeta^S_F(2s-1).\]

\subsubsection{Residues of the Borel Eisenstein series}\label{ResBortoSQ}

In the following, we relate the residue of the Borel Eisenstein series along the line $z =1$ to the Siegel Eisenstein series. Let $f \in I_{B_\H}(s,z)$ be the section defining $\mathrm{Eis}_{B_{\H}}(g,z,s)$, then
\[\mathrm{Res}_{z = 1}\mathrm{Eis}_{B_{\H}}(g,z,s) =  \sum_{\gamma'\in P_{\H}(F)\setminus \H(F)}r(f)(\gamma'g),\]
where $r(f)(g) = \mathrm{Res}_{z = 1}\left(\sum_{\gamma\in B_{\H}(F)\setminus P_{\H}(F)}f(\gamma g,s,z)\right)$. Since $B_{\H}\backslash P_{\H}\simeq B_{\GL_2}\backslash  \GL_2$, we have \[f \in I_{P_{\H}(\A)}^{\H(\A)}\left(\delta_{P_{\H}}^{\frac{1}{3}(s+1)}I_{B_{\GL_2}(\A)}^{\GL_2(\A)}\delta_{B_{\GL_2}}^z\right).\] As the residue of degenerate Eisenstein series on $\GL_2$ is constant, we get that $r(f)(g)\in I_{P_\H}(s)$. Summing up, we have obtained the following:

\begin{proposition}\label{residuesofBES}
For $f$ and $r(f)$ as above we have that $\mathrm{Res}_{z = 1}\mathrm{Eis}_{B_{\H}}(g,z,s) = {\rm Eis}_{P_{\H}}(g,s, r(f))$.
\end{proposition}

Similarly, we have a formula for the residue along $s-z=1$:
\begin{align*}
    \mathrm{Res}_{s-z = 1}\mathrm{Eis}_{B_{\H}}(g,z,s) &= \mathrm{Eis}_{Q_{\H}}(g,2z+1/2, r'(f)),
\end{align*}
where $r'(f)(g) = \mathrm{Res}_{s-z = 1}\left(\sum_{\gamma\in B_{\H}(F)\setminus Q_{\H}(F)}f(\gamma g,s,z)\right)$.

\subsection{Shalika models and a theta correspondence}\label{ss:Shalikamodelsandtheta}

The Shalika subgroup of $\G$ is defined as follows. For any $F$-algebra $R$, $S$ is given by
\begin{equation}\label{shalgrp}S(R) = \left\{\left(\begin{smallmatrix}a&b&&\\c&d&& \\ & &a&-b\\ & &-c&d\end{smallmatrix}\right) \left(\begin{smallmatrix} 1& &\alpha & x\\& 1&y&\overline{\alpha}\\ & & 1& \\ & & &1 \end{smallmatrix}\right)\;:\; \left(\begin{smallmatrix}a&b\\c&d\end{smallmatrix}\right) \in \GL_2(R), x,y\in \mathbb{G}_a(R),\;\alpha\in {\rm Res}_{E/F}\mathbb{G}_a(R) \right\}.\end{equation}
Note that the Shalika group $S$ is isomorphic to $\GL_2 N_{P_\G}$, where $N_{P_\G}$ is the unipotent radical of the Siegel parabolic $P_\G$ of $\G$, while $\GL_2$ embeds into the Levi of the Siegel parabolic $M_{P_\H}$ of $\H$ via the embedding $h \mapsto h^*= \left(\begin{smallmatrix}h& \\ & {\rm det}(h) J_2 {}^th^{-1} J_2\end{smallmatrix}\right)$. Notice that the modulus character $\delta_S$ of $S$ is trivial. 

Recall that we have fixed a non-degenerate character $\psi: F \backslash \A \to \C^\times$. Given $\eta \in E^\times$ such that $\bar{\eta} = - \eta$, we define
\[\chi_{\eta}: N_{P_\G}(F) \backslash N_{P_\G}(\A) \longrightarrow \C^{\times}, \, \left( \begin{smallmatrix}1& &\alpha&x\\&1 &y&\overline{\alpha}\\ & & 1& \\ & & &1 \end{smallmatrix}\right)  \mapsto \psi(\eta\alpha-\eta\overline{\alpha}).\]
By a slight abuse of notation, denote also by $\chi_{\eta}$ the character on $ S(F) \backslash S(\A)$ given by $ h^* n \mapsto \chi_\eta(n)$. Let $\Pi$ be a cuspidal automorphic representation of $\G(\A)$ with trivial central character.   

\begin{definition}
For a cusp form $\varphi$ in the space of $\Pi$, the unitary Shalika period is defined by 
\[ \mathcal{S}_{\varphi}^\eta(g) :=  \int_{ \A^\times \GL_2(F) \backslash \GL_2(\A)} \int_{N_{P_\G}(F) \backslash N_{P_\G}(\A)} \varphi \left( n h^* g\right) \chi_\eta^{-1}(n) d n d g.  \]
The representation $\Pi$ is said to have a Shalika model if there exist  $\varphi$ in $\Pi$ and $g \in \G(\A)$ such that $\mathcal{S}_{\varphi}^\eta(g) \ne 0$.
\end{definition}

In \cite{FurusawaMorimoto}, Furusawa and Morimoto characterized the representations $\Pi$ with unitary Shalika model in terms of the existence of a pole at $s=1$ of the partial twisted exterior square $L$-function. Namely, they showed the following.

\begin{theorem}[{\cite[Theorem 4.1]{FurusawaMorimoto}}]\label{FurusawaMmain}
Let $\Pi$ be an irreducible cuspidal automorphic representation of $\G(\A)$ with trivial central character. Then the following two conditions are equivalent: \begin{enumerate}
    \item The Shalika period $\mathcal{S}_{\varphi}^\eta(1)$ does not vanish on the space of $\Pi$.
    \item $\Pi$ is globally generic and the partial $L$-function $L^\Sigma(s,\Pi, \wedge^2)$ has a simple pole at $s=1$. 
\end{enumerate}
\end{theorem}

Motivated by this, in \cite{morimoto}, Morimoto gave a characterization of cuspidal automorphic representations of $\G(\A)$ with Shalika models by means of the theta correspondence $\Theta$ of the pair $(\mathbf{P}\H(\A)^+,\mathbf{P}\G(\A))$, with 
 \[\H(\A)^+ := \{ (g,m_g) \in \H(\A)\,:\, m_g \in N_{E/F}( \A_E^\times) \}. \]
 
\begin{theorem}[{\cite[Theorem 3.6 \& (3.6)]{morimoto}}]\label{Morimotomain}
Let $\Pi$ be an irreducible cuspidal automorphic representation of $\G(\A)$ with trivial central character. Then $\Pi$ has a Shalika model if and only if $\Pi$ appears as a subquotient of $\Theta(\sigma^+)$, for an irreducible cuspidal automorphic representation $\sigma^+$ of $\H(\A)^+$ which appears as a subrepresentation of a globally generic cuspidal automorphic representation $\sigma$ of $\H(\A)$ with trivial central character. If any of these conditions hold, we have 
\[ L^\Sigma(s,\Pi, \wedge^2) = L^\Sigma(s,\sigma, {\rm std} \otimes \chi_{E/F}) \zeta^\Sigma_F(s),\]
where $\chi_{E/F}$ is the Hecke character associated to the extension $E/F$ and $\Sigma$ is a finite set of places which contains all the ramified places of $\Pi$, $\sigma$, and $E/F$.
\end{theorem}

\subsection{The representation of the standard $L$-function of $\GSp_4$}

Let $\Pi$ be a cuspidal automorphic representation of $\G(\A)$ with trivial central character. Let $\Sigma$ be a finite set of places containing the archimedean places and all the ramified places for $\Pi$ and $E/F$. We take a cusp form $\varphi \in \Pi$, which we assume to be a pure tensor and consider the following integral:
\begin{equation*}  J(\varphi, f,   s):=\int_{\H(F)Z_{\H}(\A)\setminus \H(\A)} \mathrm{Eis}_{P_\H}(h,s,f) \varphi(h) d h,\end{equation*}
where $\mathrm{Eis}_{P_\H}(h,s,f)$ is the Siegel Eisenstein series for $\H$ introduced in \S \ref{SiegelEisSeries} associated to a factorizable standard holomorphic section $f$ unramified outside $\Sigma$. Note that this integral converges absolutely for the values of $s$ when $ \mathrm{Eis}_{P_\H}(h,s,f)$ is holomorphic, as the restriction to $\H(\A)$ of $\varphi$ is rapidly decreasing, while $\mathrm{Eis}_{P_\H}(h,s,f)$ is slowly increasing. Hence, it defines a meromorphic function. In \S \ref{sec:PropRSI}, we will show that $J(\varphi, f,   s)$ is holomorphic even at the points where the Eisenstein series $\mathrm{Eis}_{P_\H}(h,s,f)$ can have a pole. 

\begin{proposition}\label{unfoldingSiegel}
The integral $ J(\varphi, f,  s)$ unfolds to 
\[ \int_{\GL_2(\A)N_{P_\H}(\A)\setminus \H(\A)} f(h,s) \mathcal{S}_{\varphi}^{\eta}(h) d h.\]
\end{proposition}
\begin{proof}
We start by unfolding the Eisenstein series; for ${\rm Re}(s)$ big enough, we get
\[\int_{P_\H(F)Z_\H(\A)\setminus \H(\A)} f(h,s) \varphi(h) d h .\]
Collapse the sum over the unipotent radical $N_{P_\H}$ of $P_\H$ to get
\begin{align*}
\int_{M_{P_\H}(F)N_{P_\H}(\A)Z_\H(\A)\setminus \H(\A)} \int_{N_{P_\H}(F) \backslash N_{P_\H}(\A)} f(nh,s) \varphi(nh) dn d h = \int_{\cdots} f(h,s) \int_{N_{P_\H}(F) \backslash N_{P_\H}(\A)}  \varphi(nh) dn d h,
\end{align*}
as $f \in I_{P_\H}(s)$.
Now Fourier expanding over $[ N_{P_\G}/N_{P_\H} ]$, with $N_{P_\G}$ the unipotent radical of the Siegel parabolic of $\G$, the integral equals 
\[ \int_{M_{P_\H}(F)N_{P_\H}(\A)Z_\H(\A)\setminus \H(\A)} f(h,s) \sum_{\chi}\varphi_\chi(h) d h,\]
where the sum runs over the characters $ \chi : [ N_{P_\G}/N_{P_\H} ] \to \C^\times$ and where we have denoted \[ \varphi_\chi(h) = \int_{N_{P_\G}(F) \backslash N_{P_\G}(\A)} \varphi(n' h) \chi^{-1}(n') d n'. \]
\noindent Any character $\chi: [N_{P_\G}] \to \C^\times$ is of the form $n' \mapsto \psi( {\rm Tr} ( A \cdot \left(\begin{smallmatrix} \alpha & x  \\ y & \bar{\alpha} \end{smallmatrix} \right) )), $ where $A \in M_{2 \times 2}(E)$ is such that $\overline{A}^t  = J_2 A J_2$, $\psi$ is a non-trivial additive character on $F \backslash \A$, and where we have written \[ n' =  \left(\begin{smallmatrix} 1 & & \alpha & x\\ & 1 & y & \bar{\alpha} \\ & & 1 & \\ & & & 1\end{smallmatrix} \right),\; \text{ with } x,y \in \A, \alpha \in \A_E.\] If $A = (a_{i,j})$, the associated character $\chi$ is trivial on $N_{P_\H}(\A)$ if $a_{1,1}+ a_{2,2} = 0$ and $a_{1,2}=a_{2,1} = 0$. We deduce that any character $\chi: [ N_{P_\G}/N_{P_\H}] \to \C^\times$ is of the form \[n' \mapsto \psi( {\rm Tr} ( \left(\begin{smallmatrix} \eta &   \\ & -\eta \end{smallmatrix} \right) \cdot \left(\begin{smallmatrix} \alpha & x  \\ y & \bar{\alpha} \end{smallmatrix} \right) )) = \psi(\eta \alpha - \eta \bar{\alpha} ) ,\] with $\eta \in E^\times$ such that $\bar{\eta} = - \eta$. To emphasize its dependence on $\eta$, we denote this character by $\chi_\eta$.

As $M_{P_\H}$ acts on $N_{P_\G}$ by conjugation and this action preserves $N_{P_\H}$, $M_{P_\H}(F)$ acts on the space of characters of $[N_{P_\G}/N_{P_\H}]$. Indeed, if $ m  =\left(\begin{smallmatrix} g & \\ & \mu J_2 {}^tg^{-1} J_2 \end{smallmatrix} \right) \in M_{P_\H}(F)$ , \[ \chi_\eta( m n' m^{-1}) = \chi_{\eta'}(n'),\; \text{ with } \eta' = \mu^{-1} {\rm det}(g) \eta. \] 
Then, $M_{P_\H}(F)$ acts on the space of characters of $[N_{P_\G}/N_{P_\H}]$ with two orbits, the trivial one and an open one associated to any $\chi_\eta$ with $\eta \ne 0$. Using this action, we can write the integral as 
\begin{align}\label{eq:Fexpansionandaction}\int_{M_{P_\H}(F)N_{P_\H}(\A)Z_\H(\A)\setminus \H(\A)} f(h,s) \int_{[N_{P_\G}]} \varphi(n' h) dn' d h +   \int_{M_\eta(F)N_{P_\H}(\A)Z_\H(\A)\setminus \H(\A)} f(h,s)\varphi_{\chi_\eta}(h) d h,\end{align} 
where $M_\eta(F)$  denotes the stabiliser in $M_{P_\G}(F)$ of $\chi_\eta$. Precisely, $m \in M_\eta(F)$ if $\mu = {\rm det}(g)$, hence we can fix an isomorphism $\GL_2(F) \simeq M_\eta(F)$, sending $g \mapsto m(g,{\rm det}(g))$. As the first integral of \eqref{eq:Fexpansionandaction} vanishes because of cuspidality of $\varphi$ along $[N_{P_\G}]$, \[  J(\varphi, f,   s) =\int_{\GL_2(F)N_{P_\H}(\A)Z_\H(\A)\setminus \H(\A)} f(h,s)\varphi_{\chi_\eta}(h) d h.\]
Collapse the sum over $[\GL_2]$, to get 
\begin{align*}
     J(\varphi, f, s)&= \int_{\GL_2(F)N_{P_\H}(\A)Z_\H\setminus \H(\A)} f(h,s)\varphi_{\chi_\eta}(h) d h \\ &= \int_{\GL_2(\A)N_{P_\H}(\A)\setminus \H(\A)} \int_{\GL_2(F) Z_\H(\A) \backslash \GL_2(\A)} f(m(g,{\rm det}(g)) h,s)\varphi_{\chi_\eta}(m(g,{\rm det}(g)) h) d g d h \\
    &= \int_{\GL_2(\A)N_{P_\H}(\A)\setminus \H(\A)} f(h,s) \int_{\GL_2(F) Z_\H(\A) \backslash \GL_2(\A)} \varphi_{\chi_\eta}(g h) d g d h \\
    &= \int_{\GL_2(\A)N_{P_\H}(\A)\setminus \H(\A)} f(h,s) \mathcal{S}_{\varphi}^{\eta}(h) d h, 
\end{align*}
where we have used that $\delta_{P_\H}(m(g,{\rm det}(g))) =\left|\tfrac{{\rm det}(g)}{{\rm det}(g)}\right|^3=1$.
\end{proof}

\begin{corollary}\label{corostupUNF}
The integral $ J(\varphi, f ,s)$ is identically zero unless $\Pi$ is globally generic and $L^{\Sigma}(s, \Pi, \wedge^2)$ has a simple pole at $s=1$. 
\end{corollary}
\begin{proof}
If $J(\varphi, f ,s)$ is not identically zero, then the Shalika period does not vanish on $\Pi$. The result then follows from Theorem \ref{FurusawaMmain}.  
\end{proof}
Corollary \ref{corostupUNF} is coherent with the following.

\begin{lemma}\label{lemmaresstupid}
For $f$ and $r(f)$ as in  \S \ref{ResBortoSQ}, we have that \[\mathrm{Res}_{z = 1} I(\varphi,f,s,z) = J(\varphi, r(f),s).\]   
\end{lemma}
\begin{proof}
    It follows immediately from Proposition \ref{residuesofBES}.
\end{proof}

By Lemma \ref{lemmaresstupid} and Theorem \ref{mainunramified}, we obtain that 
 \begin{align}
     J(\varphi, r(f),s) &= {\rm Res}_{z=1} \left( \frac{L^\Sigma(s, \Pi, \wedge^2)L^\Sigma(z, \Pi, \wedge^2)}{\zeta_F^\Sigma(2z)\zeta_F^\Sigma(2s)\zeta_F^\Sigma(s+z)\zeta_F^\Sigma(s-z+1)}  \prod_{v \in \Sigma} I(\varphi_v,f_v, s  , z ) \right) \nonumber\\ 
     &= \frac{L^\Sigma(s, \Pi, \wedge^2)}{\zeta_F^\Sigma(2)\zeta_F^\Sigma(2s)\zeta_F^\Sigma(s+1)\zeta_F^\Sigma(s)}{\rm Res}_{z=1} \left( L^\Sigma(z, \Pi, \wedge^2)  \prod_{v \in \Sigma}
     I(\varphi_v,f_v, s  , z ) \right).\label{partialform}
 \end{align}

Let $\Sigma_\infty \subseteq \Sigma$ denote the set of archimedean places of $F$.

\begin{proposition}\label{TowardsMain2}
Given a cusp form $\varphi$ in $\Pi$, there exists a section $f$ such that
 \begin{align*}
     J(\varphi, r(f),s) = \frac{L^\Sigma(s, \Pi, \wedge^2)}{\zeta_F^\Sigma(2)\zeta_F^\Sigma(2s)\zeta_F^\Sigma(s+1)\zeta_F^\Sigma(s)} \prod_{v \in \Sigma}
     I(\varphi_v,f_v, s  , 1 ){\rm Res}_{z=1} \left( L^\Sigma(z, \Pi, \wedge^2)\right).
 \end{align*}
  Moreover, if $\Pi$ has a non trivial Shalika period there exists a cusp form $\varphi'$ in $\Pi$ and a section $f$ such that
 \begin{align*}
     J(\varphi, r(f),s) = \frac{L^\Sigma(s,\sigma, {\rm std} \otimes \chi_{E/F})}{ \zeta_F^\Sigma(2s)\zeta_F^\Sigma(s+1)} \prod_{v \in \Sigma_\infty}
     I(\varphi_v',f_v, s  , 1 ),
 \end{align*}
 with $\sigma$ a globally generic cuspidal automorphic representation of $\H(\A)$ as in Theorem \ref{Morimotomain}.
\end{proposition}

\begin{proof}
By Lemma \ref{LemmaonRamified}, we can choose $\varphi_v$ and $f_v$ such that $\prod_{v \in \Sigma - \Sigma_\infty}
     I(\varphi_v,f_v, s  , z )$ is a non-zero constant $C$ independent of $s$ and $z$ and there exists $\epsilon > 0$ such that $\prod_{v \in \Sigma_\infty}
     I(\varphi_v,f_v, s  , z )$ converges absolutely for ${\rm Re}(z)>1-\epsilon$ and ${\rm Re}(s) > {\rm Re}(z) -4$. This and \eqref{partialform} prove the first equality. 
     Now suppose that $\Pi$ has a Shalika period. By Theorems \ref{FurusawaMmain} and \ref{Morimotomain}, ${\rm Res}_{z=1} \left( L^\Sigma(z, \Pi, \wedge^2)\right) \ne 0$ and there exists a globally generic cuspidal automorphic representation $\sigma$ of $\H(\A)$ with trivial central character such that, up to enlarging $\Sigma$,
\[ L^\Sigma(s,\Pi, \wedge^2) = L^\Sigma(s,\sigma, {\rm std} \otimes \chi_{E/F}) \zeta^\Sigma_F(s).\] 
This implies that 
     \begin{align*}
     J(\varphi, r(f),s) = \frac{C\cdot L^\Sigma(s,\sigma, {\rm std} \otimes \chi_{E/F})}{\zeta_F^\Sigma(2)\zeta_F^\Sigma(2s)\zeta_F^\Sigma(s+1)} \prod_{v \in \Sigma_\infty}
     I(\varphi_v,f_v, s  , 1 )L^\Sigma(1,\sigma, {\rm std} \otimes \chi_{E/F}){\rm Res}_{z=1} \zeta_F^\Sigma(z).
 \end{align*}
For a given $v \in \Sigma_{\infty}$, define $\varphi'_v : = [C \cdot L^\Sigma(1,\sigma, {\rm std} \otimes \chi_{E/F}){\rm Res}_{z=1} \zeta_F^\Sigma(z)]^{-1}\zeta_F^\Sigma(2) \cdot \varphi_v$ and let $\varphi_v' = \varphi_v$ otherwise. Then we have 
\begin{align*}
     J(\varphi, r(f),s) = \frac{L^\Sigma(s,\sigma, {\rm std} \otimes \chi_{E/F})}{ \zeta_F^\Sigma(2s)\zeta_F^\Sigma(s+1)} \prod_{v \in \Sigma_\infty}
     I(\varphi_v',f_v, s  , 1 ),
 \end{align*}
as desired.
\end{proof}

\noindent Let \begin{equation*}  J^*(\varphi,   s):=\int_{\H(F)Z_{\H}(\A)\setminus \H(\A)} E_{P_\H}^*(h,s) \varphi(h) d h,\end{equation*} where $E_{P_\H}^*(h,s)$ denotes the normalized Siegel Eisenstein series of \eqref{normalizedSES}. By Proposition \ref{TowardsMain2}, after normalizing the section $r(f)$, we get the following.
\begin{theorem}\label{Thisismain2}
Given a cusp form $\varphi$ in $\Pi$, there exists a cusp form $\varphi'$ and section $f$ such that \[J^*(\varphi,   s) = L^\Sigma(s,\sigma, {\rm std} \otimes \chi_{E/F}) \cdot \prod_{v \in \Sigma_\infty}
     I(\varphi_v',f_v, s  , 1 ). \]
\end{theorem}
    
\subsection{Properties of the Rankin--Selberg integrals}\label{sec:PropRSI}

In the following, we show that the integral $J^*(\varphi,s)$ is holomorphic for all $s$.

\subsubsection{Vanishing of a zeta integral}

In view of Proposition \ref{Ikeda} and Proposition \ref{residuesofBES}, we are interested in calculating the integral
\begin{equation}\label{eq6} J'(\varphi,f',s):=\int_{\H(F)Z_{\H}(\A)\setminus \H(\A)} {\rm Eis}_{Q_\H}(h,s,f') \varphi(h) d h,\end{equation}
where $\varphi$ is a cusp form in $\Pi$.

\begin{proposition}\label{vanishingofthemaledetto}
The integral $J'(\varphi,f',s)$ is identically zero.
\end{proposition}

\begin{remark}
Notice that if the induced datum of the Klingen Eisenstein series is not degenerate,  the corresponding integral calculates the $\wedge^2 \times {\rm Std}$ $L$-function on $\GU_{2,2} \times \GL_2$, as we show in forthcoming work.
\end{remark}

\begin{proof}[Proof of Proposition \ref{vanishingofthemaledetto}]
We start by unfolding the Eisenstein series. For big enough ${\rm Re}(s)$, we get
\[\int_{{Q_\H}(F)Z_\H(\A)\setminus \H(\A)} f'(h,s)) \varphi(h) d h .\]
Collapse the sum over the unipotent radical $N_{Q_\H}$ of ${Q_\H}$ to get
\begin{align*}
 \int_{M_{Q_\H}(F)N_{Q_\H}(\A)Z_\H(\A)\setminus \H(\A)} f'(h,s) \int_{N_{Q_\H}(F) \backslash N_{Q_\H}(\A)}  \varphi(nh) dn d h.
\end{align*}
We now Fourier expand over $[ N_{Q_\G}/N_{Q_\H} ]$. The integral equals 
\[ \int_{M_{Q_\H}(F)N_{Q_\H}(\A)Z_\H(\A)\setminus \H(\A)} f'(h,s) \sum_{\chi}\varphi_\chi(h) d h,\]
where the sum runs over the characters $ \chi : [ N_{Q_\G}/N_{Q_\H} ] \to \C^\times$. As shown in the proof of Proposition \ref{unfolding}, any character $\chi: [ N_{Q_\G}/N_{Q_\H} ] \to \C^\times$ is of the form $\chi_{\alpha,\beta}:n \mapsto  \psi( {\rm Tr}_{E/F} (\alpha x + \beta y))$, with $\alpha, \beta \in E$ such that $\bar{\alpha} = - \alpha$ and  $\bar{\beta} = - \beta$. Again from the proof of Proposition \ref{unfolding}, we know that $M_{Q_\H}(F)$ acts on the space of characters of $[N_{Q_\G} / N_{Q_\H}]$ with two orbits, the trivial one and an open one, with a representative given by the character $\chi_{\alpha,0}$ with $\bar{\alpha} = - \alpha \ne 0$. By cuspidality of $\varphi$ along $N_{Q_\G}$, we have
\begin{align}\label{eq:Fexpansionandaction1} J'(\varphi,f',s) = \int_{M_\alpha(F)N_{Q_\H}(\A)Z_\H(\A)\setminus \H(\A)} f'(h,s)\varphi_{\chi_{\alpha,0}}(h) d h,\end{align} 
where $M_\alpha(F)= \left \{ \left(\begin{smallmatrix} a &   &   &  \\ & a &  b &   \\ & & d &   \\ & & & d\end{smallmatrix} \right)\; \right \} = L_\alpha(F) N_\alpha(F) $  denotes the stabiliser in $M_{Q_\H}(F)$ of $\chi_{\alpha,0}$, with Levi part (resp. unipotent part) $L_\alpha$ (resp. $N_\alpha$). Collapse the sum over $[N_\alpha]$ to get 
\begin{align*}
     J'(\varphi,f',s)&= \int_{L_\alpha(F)N_\alpha(\A) N_{Q_\H}(\A)Z_\H\setminus \H(\A)} f'(h,s) \int_{[N_\alpha]} \varphi_{\chi_{\alpha,0}}(n h) d n d h.
\end{align*}
Notice that the inner integral $\int_{[N_\alpha]} \varphi_{\chi_{\alpha,0}}(n h) d n$ equals the Fourier coefficient of $\varphi$ with respect to the unipotent radical of the Borel subgroup of $\G$ and the choice of a character on it which is degenerate (as it's only supported on the entry (1,2)). As a consequence,  $\int_{[N_\alpha]} \varphi_{\chi_{\alpha,0}}(n h) d n$ contains, as an inner integral, the period of $\varphi$ over the unipotent radical of the Siegel parabolic of $\G$, which vanishes by cuspidality of $\varphi$.
This completes the proof.
\end{proof} 

\begin{remark}
Proposition \ref{vanishingofthemaledetto} has potential applications to the study of the archimedean regulator of the cohomology classes introduced in \cite{CauchiGU22}. Indeed, we expect the vanishing of this integral to imply that the image via the archimedean regulator of the motivic incarnation of the classes of \cite[Definition 6.6]{CauchiGU22} contributes entirely to Eisenstein cohomology. 
\end{remark}

\subsubsection{Periods over $\GSp_4$}

For any cusp form $\varphi$ on $\G(\A)$, we consider \[  \varphi^\H(g) := \int_{\H(F)Z_\H(\A)\setminus \H(\A)}   \varphi(hg) d h.\] 

\begin{lemma}\label{lemmaontheperiod}
Let $\Pi$ be a globally generic cuspidal automorphic representation of $\mathbf{P}\G(\A)$. For any cusp form $\varphi$ in $\Pi$ and $g \in \G(\A)$, the period $\varphi^\H(g)=0$. 
\end{lemma}

\begin{proof}
Since $\Sp_4$ is a normal subgroup of $\H$, the result is implied by showing that, for any $\varphi$ in $\Pi$ and $g \in \G(\A)$, \[  \varphi^{\Sp_4}(g) := \int_{[\Sp_4]}  \varphi(hg) d h=0.\]
We show the latter by contradiction. Suppose that there exists $\varphi$ in $\Pi$ and $g \in \G(\A)$ for which $\varphi^{\Sp_4}(g) \ne 0$. This defines a non-trivial $\Sp_4$-functional on $\Pi$.  Let $v$ be an unramified place for $\Pi$ which splits in $E$. We can further assume that $g_v=1$. Since $v$ splits, we can identify $\Pi_v$ with an unramified generic representation $\pi_v$ of $\GL_4(F_v)$ with trivial central character. As the linear span $R_v$ of the $\textbf{P}\G(F_v)$-right translates of $\varphi$ is isomorphic to $\pi_v$, the map \[ R_v \ni \varphi' \mapsto \left( g'  \mapsto \int_{[\Sp_4]}  \varphi'(hgg' ) d h \right) \] defines an injective intertwining map in ${\rm Hom}_{\GL_4(F_v)}(\pi_v, {C}_c^\infty(\Sp_4(F_v) \backslash \GL_4(F_v)))$. This contradicts the genericity of $\pi_v$ (\textit{cf}. \cite[Theorem 3.2.2]{HeumosRallis}). The Lemma follows. 
\end{proof}

\subsubsection{Location of poles}

\begin{theorem}\label{polesofintegral}
Let $\Pi$ be a globally generic cuspidal automorphic representation of $\G(\A)$, with trivial central character and let $\varphi$ be a cusp form in $\Pi$. The integral $J^*(\varphi,s)$ is holomorphic for all $s$. In particular, if $\Pi$ is the small theta lift of a $\H(\A)^+$-factor of a globally generic cuspidal automorphic representation $\sigma$ of $\H(\A)$, then $L^\Sigma(s,\sigma, {\rm std} \otimes \chi_{E/F})$ is holomorphic for all $s$, with $\Sigma$ a finite set of places containing all the ramified places of $\Pi$, $\sigma$, and $E/F$. 
\end{theorem}

\begin{proof}
Recall that the normalized Eisenstein series $E_{P_\H}^*(g,s)$ achieves poles at $s=1,2$ of order at most one, thus $J^*(\varphi ,s)$ may have a pole at these two values. 

At $s=1$, the residue of the integral equals to 
\begin{align*}
{\rm Res}_{s=1} J^*(\varphi,s) &=    {\rm Res}_{s=1} \int_{\H(F)Z_\H(\A)\setminus \H(\A)} E^*_{P_\H}(h,s) \varphi(h) d h \\
&= \int_{\H(F)Z_\H(\A)\setminus \H(\A)} {\rm Res}_{s=1}E^*_{P_\H}(h,s) \varphi(h) d h \\
&= C \cdot  \int_{\H(F)Z_\H(\A)\setminus \H(\A)}   {\rm Eis}_{Q_\H}( g , 1/2,f_{1/2}') \varphi(h) d h \\
&= C \cdot  J'(\varphi, f_{1/2}',1/2),
\end{align*}
where the third equality follows from Proposition \ref{Ikeda} and $C=\frac{{\rm Res}_{s=1}(\zeta^\Sigma_F(s))\zeta^\Sigma_F(3)}{2 \zeta^\Sigma_F(2)}$. The latter integral vanishes by Proposition \ref{vanishingofthemaledetto}, hence $J^*(\varphi,s)$ is holomorphic at $s=1$.

At $s=2$ there exists $f \in I_{P_\H}(s)$ such that ${\rm Res}_{s=2} E_{P_\H}^*(g,s)$ is constant. Therefore 
\begin{align*}
{\rm Res}_{s=2} J^*(\varphi,s) &= \int_{\H(F)Z_\H(\A)\setminus \H(\A)} {\rm Res}_{s=2}E^*_{P_\H}(h,s) \varphi(h) d h \\
&= C \cdot  \int_{\H(F)Z_\H(\A)\setminus \H(\A)}   \varphi(h) d h .
\end{align*}
Hence, $J^*(\varphi,s)$ achieves a pole at $s=2$ if $\varphi^{\H}(1) \ne 0$. However, the latter is always zero because of  Lemma \ref{lemmaontheperiod}. This proves the first statement. The second one follows from this and from Theorem \ref{Thisismain2}.
\end{proof}

\begin{remark}
    Since the theta lift of $\sigma$ to ${\rm PGSO}_{4,2} \simeq \mathbf{PG}$ is non-zero and cuspidal by hypothesis, Rallis tower property implies that the theta lift of $\sigma$ to ${\rm GSO}_{3,1}$ and ${\rm GSO}_{2,0}$ is zero. Thus, by using \cite[Theorem 7.2.5]{KudlaRallis}, one could already deduce that the partial $L$-function $L^\Sigma(s,\sigma, {\rm std} \otimes \chi_{E/F})$ is entire. 
\end{remark}

\bibliographystyle{acm}

\bibliography{MultivariableGU22}

\end{document}